\documentclass[11pt]{article}
\usepackage{graphicx} 

\usepackage{fancyhdr}
\usepackage{isomath}
\usepackage{mathtools} 
\usepackage{amsbsy}
\usepackage{amssymb,amsthm}
\usepackage{amscd,amsfonts}
\usepackage{bigints}
\usepackage{graphicx}
\usepackage{verbatim}
\usepackage{euscript}
\usepackage{alltt}
\usepackage{stmaryrd}
\usepackage{relsize}
\usepackage{enumerate}
\usepackage{url}
\usepackage[stable]{footmisc}
\usepackage{breakurl}
\usepackage{hyperref}
\usepackage{comment}
\usepackage[numbers]{natbib}
\usepackage[font=small,labelfont=bf]{caption}
\usepackage[font=small,labelfont=bf]{subcaption}
\usepackage{float}
\usepackage{mwe}
\usepackage{algorithm}
\usepackage{algpseudocode}
\usepackage{xcolor}
\usepackage[super]{nth}
\usepackage{soul}
\usepackage{centernot}
\usepackage{bm}

\DeclareGraphicsExtensions{.eps,.pdf}
\usepackage{amsmath}
\usepackage{amsbsy}
\usepackage{amssymb}
\usepackage{amscd}
\usepackage{amsfonts}

\newcommand{\R}{\mathbb R}

\newcommand{\beq}{\begin{equation}}
\newcommand{\eeq}{\end{equation}}
\newcommand{\beqs}{\begin{eqnarray}}
\newcommand{\eeqs}{\end{eqnarray}}
\newcommand{\beql}{\begin{equation} \label}
\newcommand{\half}{\frac{1}{2}}

\newcommand{\calB}{{\cal B}}

\newcommand{\calF}{{\cal F}}

\newtheorem{proposition}{Proposition}[section]
\newtheorem{remark}{Remark}[section]

\newcommand{\p}{\partial}

\usepackage[margin=1in]{geometry}

\theoremstyle{definition}
\newtheorem{definition}{Definition}[section]

\newcommand{\dee}{\mathcal{D}}
\newcommand{\scl}{\mathcal{L}}

\date{}
\begin{document}
\title{A Hidden Convexity of Nonlinear Elasticity}

\author{Siddharth Singh\thanks{Department of Civil \& Environmental Engineering, Carnegie Mellon University, Pittsburgh, PA 15213, email: ssingh3@andrew.cmu.edu.} $\qquad$  Janusz Ginster\thanks{Institut f\"ur Mathematik, Humboldt-Universit\"at zu Berlin, Unter den Linden 6, 10117 Berlin, Germany, email: ginsterj@hu-berlin.de.} $\qquad$ Amit Acharya\thanks{Department of Civil \& Environmental Engineering, and Center for Nonlinear Analysis, Carnegie Mellon University, Pittsburgh, PA 15213, email: acharyaamit@cmu.edu.}}

\maketitle
\begin{abstract}
\noindent A technique for developing convex dual variational principles for the governing PDE of nonlinear elastostatics and elastodynamics is presented. This allows the definition of notions of a variational dual solution and a dual solution corresponding to the PDEs of nonlinear elasticity, even when the latter arise as formal Euler-Lagrange equations corresponding to non-quasiconvex elastic energy functionals whose energy minimizers do not exist. This is demonstrated rigorously in the case of elastostatics for the Saint-Venant Kirchhoff material (in all dimensions), where the existence of variational dual solutions is also proven. The existence of a variational dual solution for the incompressible neo-Hookean material in 2-d is also shown. Stressed and unstressed elastostatic and elastodynamic solutions in 1 space dimension corresponding to a non-convex, double-well energy are computed using the dual methodology. In particular, we show the stability of a dual elastodynamic equilibrium solution for which there are regions of non-vanishing length with negative elastic stiffness, i.e.~non-hyperbolic regions,  for which the corresponding primal problem is ill-posed and demonstrates an explosive `Hadamard instability;' this appears to have implications for the modeling of physically observed softening behavior in macroscopic mechanical response.

\end{abstract}

\section{Introduction}
The goal of this paper is to demonstrate a technique for obtaining solutions of the governing PDE of nonlinear elasticity. The elasticity may be of Cauchy type, with stress not necessarily arising as a derivative of an energy density function of the deformation gradient. The methodology is of a `dual' variational nature, transforming the given `primal' PDE problem written in first-order form, in terms of the physical deformation, deformation gradient, and velocity fields, to a dual variational and PDE problem in terms of dual Lagrange multiplier fields. The fundamental insight is to treat the primal PDE under consideration as constraints and invoke a more-or-less arbitrarily designable strictly convex, auxiliary potential to be optimized. Then, a dual variational principle for the Lagrange multiplier (dual) fields can be designed involving a dual-to-primal (DtP) mapping (i.e.~an adapted change of variables),  with the special property that its Euler-Lagrange equations are exactly the primal PDE system, interpreted as equations for the dual fields using the DtP mapping, and this, even though the primal system may not have the required symmetries necessary to be the Euler-Lagrange equations of any objective functional of the primal fields alone. As we show, the dual variational principle is convex, with its formal Euler-Lagrange system being locally degenerate elliptic, regardless of the monotonicity properties of the primal PDE system. Thus, to the extent that coercivity of the dual functional can be established, existence of minimizers of the dual functional, our variational dual solutions that are formally consistent with weak solutions of the primal PDE system, can be established. Since these ideas do not even depend on the existence of an energy density function for the primal problem, the issue of existence of minimizers of a primal energy functional is not relevant to our technique; instead it works with the structure of the stress response function and the equations of equilibrium for the primal model of elasticity invoked. Even in the case of hyperelasticity, if the
primal energy functional is not lower-semicontinuous then classical variational techniques cannot be used to prove existence of minimizers. In this case, the classical approach is to consider the relaxation of the energy functional. However, minimizers of the relaxed energy functional typically do not satisfy the original PDE system.

The question of dual variational principles for nonlinear elasticity - in the form of the possibility of an extremum principle of a complementary energy - has been explored in \cite{zubov1970stationary, koiter1973principle, de1972new, berdichevskii1979dual, knops2003uniqueness}; the question of invertibility of the first Piola-Kirchhoff stress as a function of the deformation gradient (or the question of expressing the right Cauchy-Green tensor in terms of the first Piola Kirchhoff stress) for physically realistic elastic response functions has been a fundamental obstacle in these earlier studies. Our concerns are different from these earlier works and not focused on a complementary energy principle - as one difference, we require our dual solution to recover as its Euler-Lagrange equations the primal system consisting of compatibility and equilibrium (in the case of statics), not by imposing the latter as a feature of the class of dual fields under consideration. Also, there is no analog in the earlier works of a `free' choice of an auxiliary potential that we crucially exploit. Our work builds on ideas from \cite[Sec.~6.1]{acharya2022action}-\cite{action_3, dual_cont_mech_plas}, and connects to the ideas of `hidden convexity in nonlinear PDE' advanced by Y.~Brenier \cite{brenier_book, brenier2018initial}.

An outline of the paper is as follows: in Secs.~\ref{sec:elastostat}-\ref{sec:deg_ell} we lay out the formal calculations motivating and explaining the details of our approach. In Sec.~\ref{sec:var_anal} these ideas are made rigorous in the context of nonlinear elastostatics. Sec.~\ref{sec:compute} contains computational results demonstrating various features of the developed methodology. Sec.~\ref{sec:concl} contains some concluding remarks.

\section{Dual nonlinear Elastostatics}\label{sec:elastostat}
Consider the problem of nonlinear elastostatics with boundary conditions of place and dead loads (for simplicity):
\begin{subequations}
    \label{eq:elastostat}
    \begin{align}
        \p_j P_{ij}(F) + b_i & = 0 \ \mbox{in} \ \Omega  \label{eq:equil_1}\\
        \p_j y_i - F_{ij} & = 0  \ \mbox{in} \ \Omega  \label{eq:x-F}\\
        P_{ij}n_j = p_i  \ \mbox{on} \ \p\Omega_p \ &; \qquad  y_i = y^{(b)}_i  \ \mbox{on} \ \p \Omega_y, \quad \p \Omega = \p \Omega_p \cup \Omega_y \nonumber.
    \end{align}
\end{subequations}
In the above, $P$ is a given tensor valued function of invertible tensors $F$ (delivering the First Piola-Kirchhoff stress tensor for a prescribed deformation gradient), and all derivatives are w.r.t. rectangular Cartesian coordinates on a fixed reference configuration $\Omega$ (\textit{$P$ is not assumed to necessarily be a gradient of a scalar valued function on the space of invertible tensors}). The only restrictions on $P$ we require are that it be sufficiently smooth in its argument
and that it be of the form 
\[
P(F) = F \,S\!\left(F^TF \right)
\]
for $S$ being any arbitrary symmetric tensor valued function of a symmetric tensor; this allows frame-indifference to be satisfied. Furthermore, $b$ is the specified body force density per unit volume of the reference configuration field and \textit{it is not required that the latter arise from a potential}.

Define the pre-dual functional by forming the scalar products of \eqref{eq:equil_1} and \eqref{eq:x-F} with the dual fields $\lambda$ and $\mu$, respectively, integrating by parts, and imposing the prescribed b.c.s:
\begin{equation}
    \label{eq:predual}
    \begin{aligned}
        \widehat{S}_H[y,F,\lambda, \mu] & = \int_\Omega - P_{ij}\big|_F \p_j \lambda_i + \lambda_i b_i \,dx  + \int_\Omega (- y_i \p_j \mu_{ij} - F_{ij} \mu_{ij})\, dx + \int_\Omega  H(y, F;x) \, dx \\
    & \quad + \int_{\p \Omega_p} p_i \lambda_i \, da  \ + \int_{\p \Omega_y} y^{(b)}_i \mu_{ij} n_j \, da.
    \end{aligned}    
\end{equation}
Define
\begin{equation}
    \label{eq:defs}
    \begin{aligned}
        U := (y, F); \qquad D := (\lambda, \mu); \qquad \dee := (\nabla \lambda, \lambda, \nabla \mu, \mu);\\
        \scl_H(U,\dee,x) := - P_{ij}\big|_F \p_j \lambda_i + \lambda_i b_i - y_i \p_j \mu_{ij} - F_{ij} \mu_{ij} + H(y,F;x),
    \end{aligned}
\end{equation}
and require the choice of the potential $H$ to be such that it facilitates the existence of a function
\[
U = U_H(\dee, x)
\]
which satisfies
\begin{equation}\label{eq:dLdU}
\frac{\p \scl_H}{\p U} \left(U_H(\dee, x),\dee,x \right) = 0 \qquad \forall \ (\dee,x).
\end{equation}
When such a `change of variables' mapping, $U_H = \left( y^{(H)}, F^{(H)}\right)$, exists, defining the \textit{dual} functional as
\begin{equation*}
        \begin{aligned}
        & S_H[D] := \widehat{S}_H \left[ y^{(H)}, F^{(H)}, \lambda, \mu \right] = \int_\Omega \scl_H\left(U_H(\dee,x), \dee, x \right) \, dx + \int_{\p \Omega_p} p_i \lambda_i \, da  \ + \int_{\p \Omega_y} y^{(b)}_i \mu_{ij} n_j \, da\\
        & \mbox{with $\lambda$ specified on $\p \Omega \backslash \p \Omega_p$ and $\mu$ specified on $\p \Omega \backslash \p \Omega_y$},
    \end{aligned}    
\end{equation*}
and noting \eqref{eq:dLdU}, the first variation of $S_H$ (about a state $D$ in the direction $\delta D$, the latter constrained by $\delta \lambda = 0$ on $\p \Omega \backslash \p \Omega_p$ and $\delta \mu = 0$ on $\p \Omega \backslash \p \Omega_y$), is given by
\[
\delta S_H = \bigintsss_\Omega \frac{ \p \scl_H}{\p \dee} \left(U_H(\dee,x), \dee, x \right) \cdot \delta \dee \, dx + \int_{\p \Omega_p} p_i \delta \lambda_i \, da  \ + \int_{\p \Omega_y} y^{(b)}_i \delta \mu_{ij} n_j \, da.
\]

Noting, now, that $\scl_H$ is necessarily affine in $\dee$, its second argument, it can be checked that \textit{the Euler-Lagrange equations and natural boundary conditions of the dual functional $S_H$ are exactly the system \eqref{eq:elastostat} with $U$ substituted by $U_H(\dee,x)$}.

It is this simple idea that we exploit to develop variational principles for (Cauchy) Elasticity.

\subsection{Quadratic nonlinearity about arbitrarily fixed base states}
In this section we implement the generalities above in a context that allows to demonstrate an explicitly written out dual functional. This is only for conceptual simplicity in conveying ideas; in fact, the examples we successfully compute in Sec.~\ref{sec:compute} will be more general than the assumptions made here.

For solutions of \eqref{eq:elastostat} `close' to an arbitrarily chosen (and fixed thereafter) pair of fields we refer to as a \textit{base state} given by
\[
x \mapsto \bar{y}(x); \qquad x \mapsto \nabla \bar{y} (x) =: \bar{F}(x),
\]
we assume that the nominal stress fields, $x \mapsto P(x)$, for all candidate pairs $(y,F)$ that belong to the domain of the functional $\widehat{S}$ are well-represented by a quadratic expansion around $(\bar{y}, \bar{F})$: i.e.
\begin{equation}\label{eq:quad_strs}
\begin{aligned}
     P_{ij}(F(x)) & = \bar{P}_{ij}\left(\bar{F}(x) \right) + A_{ijkl}|_{\bar{F}(x)} \left( F_{kl}(x) - \bar{F}_{kl}(x)\right) \\
     & \quad + \frac{1}{2} B_{ijklmn}|_{\bar{F}(x)} \left( F_{kl}(x) - \bar{F}_{kl}(x)\right)\left( F_{mn}(x) - \bar{F}_{mn}(x)\right),
\end{aligned}
\end{equation}
where $\bar{P}, A, B$ are given functions from the space of second order tensors to the space of second, fourth, and sixth order tensors, respectively. We assume the symmetries $B_{ijklmn} = B_{ijmnkl}$, without loss of generality. The coefficients $\bar{P}, A, B$ may be interpreted as those of a truncated Taylor expansion of the nominal stress response function about $\bar{F}$, but not necessarily. In the following, we will assume these functions to be bounded.


We note that the assumption on the form of the stress allows for nominal stress fields that access `monotone increasing' and `monotone decreasing' parts of the `stress-strain' response of the elastic material at different spatial locations of the domain.

We now choose a `shifted quadratic' form for the potential $H$:
\[
Q(y,F;x) = \frac{1}{2} \left( c_y \left|y - \bar{y}(x) \right|^2 + c_F \left|F - \bar{F}(x)\right|^2  \right),
\]
where 
\[
c_y, c_F \gg 1
\]  
are real-valued scalars. Then, the implementation of \eqref{eq:dLdU} in this context means algebraically solving for $\left(y^{(Q)}, F^{(Q)}\right)$ in terms of $(x, \nabla \lambda, \nabla \mu, \mu)$ from the \textit{dual-to-primal} (DtP) mapping equations
\begin{equation*}
    \begin{aligned}
        \frac{\p \scl_Q}{\p y_i} :& \qquad - \p_j \mu_{ij} + c_y \left(y_i^{(Q)} - \bar{y}_i(x) \right) = 0 \\
     \frac{\p \scl_Q}{\p F_{rs}} :& \qquad - \p_j \lambda_i \left( A_{ijrs}|_{\bar{F} (x)} + B_{ijklrs}|_{\bar{F}(x)} \left( F_{kl}^{(Q)} - \bar{F}_{kl}(x) \right) \right)  - \mu_{rs} + c_F \left(F_{rs}^{(Q)} - \bar{F}_{rs}(x) \right) = 0.
    \end{aligned}
\end{equation*}
which can be rearranged as the following relations (suppressing the explicit $x$ dependence of the base state for brevity)
\begin{subequations}
    \label{eq:quad_dldu}
    \begin{align}
        c_y\left( y^{(Q)}_i - \bar{y}_i\right) & = \p_j \mu_{ij}  \label{eq:map_y} \\
        c_F \, \mathbb{K}_{rskl}|_{\bar{F}} \left( F_{kl}^{(Q)} - \bar{F}_{kl} \right) & = \mu_{rs} + \p_j \lambda_i A_{ijrs}|_{\bar{F}} \label{eq:map_F}\\
        \mbox{where} \qquad \mathbb{K}_{rskl}\big|_{(\bar{F}, \nabla \lambda)} &:= \delta_{kr} \delta_{ls} - \frac{1}{c_F} \p_j \lambda_i B_{ijklrs}|_{\bar{F}}.  \quad \mbox{is affine in} \ \nabla \lambda.   \label{eq:K}
    \end{align}
\end{subequations}

Then
\begin{equation*}
    \begin{aligned}
    & \widehat{S}_Q \left[y^{(Q)}, \lambda, F^{(Q)}, \mu \right]  =  \\
    &  \quad \int_\Omega -\p_j \lambda_i \biggl(  \bar{P}_{ij}\left(\bar{F} \right) + A_{ijkl}|_{\bar{F}} \left( F_{kl}^{(Q)} - \bar{F}_{kl}\right) 
    + \frac{1}{2} B_{ijklmn}|_{\bar{F}} \left( F_{kl}^{(Q)} - \bar{F}_{kl}\right)\left( F_{mn}^{(Q)} - \bar{F}_{mn}\right)
    \biggr) \, dx \\
    & + \int_\Omega \left( - \mu_{ij} \left( F_{ij}^{(Q)} - \bar{F}_{ij}\right) - \mu_{ij} \bar{F}_{ij} \right) \, dx \quad + \int_\Omega \left( - \left( y_i^{(Q)} - \bar{y}_i \right) \p_j \mu_{ij} - \bar{y}_i \p_j \mu_{ij} \right) \, dx \\
    & + \int_\Omega \frac{1}{2} \left( c_y \left|y^{(Q)} - \bar{y} \right|^2 + c_F \left|F^{(Q)} - \bar{F}\right|^2  \right) \, dx  \quad + \int_{\p \Omega_p} p_i \lambda_i\, da \quad + \int_{\p \Omega_y} y^{(b)}_i \mu_{ij} n_j \, da + \int_\Omega \lambda_i b_i \, dx.
    \end{aligned}
\end{equation*}

\textit{For notational simplicity, in the following computations, we will omit mention of the superscript $^{(Q)}$}.

We obtain the dual functional through the following computation utilizing \eqref{eq:quad_dldu}:
\begin{equation}
    \begin{aligned}
         S_Q[\lambda, \mu] & = \int_\Omega \left( - \bar{P}_{ij}\big|_{\bar{F}} \p_j \lambda_i  -  \bar{F}_{ij}\mu_{ij} - \bar{y}_i \p_j \mu_{ij} + \lambda_i b_i \right) \, dx  \quad - \int_\Omega \Big( \mu_{rs} + \p_j \lambda_i A_{ijrs}\big|_{\bar{F}} \Bigr) \left( F_{rs} - \bar{F}_{rs} \right) \, dx\\
          &  \quad + \frac{1}{2} \int_\Omega \Bigl( c_F \delta_{kr} \delta_{ls} - \p_j \lambda_i B_{ijklrs}\big|_{\bar{F}} \Bigr) \left( F_{kl} - \bar{F}_{kl} \right) \left( F_{rs} - \bar{F}_{rs} \right) \, dx\\
          & \quad - \int_\Omega \p_j \mu_{ij} (y_i - \bar{y}_i) \, dx + \frac{1}{2} \int_\Omega c_y (y_i - \bar{y}_i) (y_i - \bar{y}_i) \, dx \quad + \int_{\p \Omega_p} p_i \lambda_i\, da \quad + \int_{\p \Omega_y} y^{(b)}_i \mu_{ij} n_j \, da \\
          & = \int_\Omega \left( - \bar{P}_{ij}\big|_{\bar{F}} \p_j \lambda_i  -  \bar{F}_{ij}\mu_{ij} - \bar{y}_i \p_j \mu_{ij} + \lambda_i b_i \right) \, dx  \quad + \int_{\p \Omega_p} p_i \lambda_i\, da \quad + \int_{\p \Omega_y} y^{(b)}_i \mu_{ij} n_j \, da \\
          & \quad - \frac{1}{2} \int_\Omega \Big( \mu_{rs} + \p_j \lambda_i A_{ijrs}\big|_{\bar{F}} \Bigr) \left( F_{rs} - \bar{F}_{rs} \right) \, dx \quad - \frac{1}{2} \int_\Omega \p_j \mu_{ij} (y_i - \bar{y}_i) \, dx,
    \end{aligned}
\end{equation}
and using \eqref{eq:quad_dldu} again, \textit{a dual functional for nonlinear (Cauchy) elastostatics} is given by
\begin{equation}\label{eq:dual_elasotstat}
    \begin{aligned}
        & S_Q[\lambda, \mu]  = \\
        & \qquad - \frac{1}{2} \int_\Omega \biggl[  \Big( \mu_{rs} + \p_j \lambda_i A_{ijrs}\big|_{\bar{F}} \Bigr) \frac{1}{c_F} \left( \mathbb{K} \big|_{(\bar{F},\nabla \lambda)}^{-1} \right)_{rskl} \Big( \mu_{kl} + \p_j \lambda_i A_{ijkl}\big|_{\bar{F}} \Bigr)  \quad + \frac{1}{c_y} \p_j \mu_{ij} \p_k \mu_{ik}  \biggr] \, dx\\
        & \qquad + \int_\Omega \left( - \bar{P}_{ij}\big|_{\bar{F}} \p_j \lambda_i  -  \bar{F}_{ij}\mu_{ij} - \bar{y}_i \p_j \mu_{ij} + \lambda_i b_i \right) \, dx \quad + \int_{\p \Omega_p} p_i \lambda_i\, da \quad + \int_{\p \Omega_y} y^{(b)}_i \mu_{ij} n_j \, da.
    \end{aligned}
\end{equation}
\subsection{Formal second variation and stability of dual solutions}\label{sec:stab}
From the DtP equations \eqref{eq:quad_dldu}, a primal pair $\left( y^{(Q)}, F^{(Q)} \right)$ corresponding to the dual fields 
\[
\bar{D} := (x \mapsto \lambda(x) = 0, x \mapsto \mu(x) = 0)
\]
is given by $(x \mapsto \bar{y}(x), x \mapsto \bar{F}(x))$. 

We are now interested in obtaining the second variation of the dual functional $S_Q$ about the dual state $\bar{D}$, along sufficiently smooth perturbations $( x \mapsto \delta D(x)\,, \, x \mapsto dD(x) )$ for $c_F, c_y  \gg 1$, as required.

We make the (physically realistic) assumption that the nominal stress response of the material is such that $A, B$ are bounded functions on the space of second order tensors.

Noting
\begin{equation*}
    \begin{aligned}
        \mathbb{K}_{rskl}\big|_{(\bar{F}, \nabla \lambda)} = \delta_{kr} \delta_{ls} - \frac{1}{c_F} \p_j \lambda_i B_{ijrskl}\big|_{\bar{F}} \ & ; \qquad  \mathbb{K}_{rskl}\big|_{(\bar{F}, \nabla \lambda = 0)} = \delta_{kr} \delta_{ls} =: \mathbb{I}_{rskl}\\
        \mathbb{K} \mathbb{K}^{-1}  = \mathbb{I} & \Longrightarrow \frac{\p \mathbb{K}^{-1}}{\p \nabla \lambda} = - \mathbb{K}^{-1} \frac{\p \mathbb{K}}{\p \nabla \lambda} \mathbb{K}^{-1}\\
       \left. \frac{\p \mathbb{K}}{\p \nabla \lambda}\right|_{(\bar{F}, \nabla \lambda)} & = - \frac{1}{c_F} B\big|_{\bar{F}} \\
       \mathbb{K}_{rskl}\big|_{(\bar{F}, \nabla \lambda)}& = \mathbb{K}_{klrs}\big|_{(\bar{F}, \nabla \lambda)},
    \end{aligned}
\end{equation*}
the first variation of $S_Q$, at $D$ and in the direction $\delta D$, is given by
\begin{equation*}
    \begin{aligned}
        & \delta S_Q \big|_{\delta D} [D] = \\
        & \qquad - \int_{\Omega} \biggl( \frac{1}{c_F} \left( \mathbb{K} \big|_{(\bar{F},\nabla \lambda)}^{-1} \right)_{rskl} \left(\mu_{rs} + \p_j \lambda_i A_{ijrs}\big|_{\bar{F}} \right) \left(\delta \mu_{kl} + \p_j \delta \lambda_i A_{ijkl}\big|_{\bar{F}} \right) \ + \ \frac{1}{c_y} \p_j \mu_{ij} \p_k \delta \mu_{ik} \biggr)\, dx \\
        & \qquad + \int_\Omega \biggl( \frac{1}{2} \left(\mu_{rs} + \p_j \lambda_i A_{ijrs}\big|_{\bar{F}} \right) \left(\mu_{kl} + \p_j \lambda_i A_{ijkl}\big|_{\bar{F}} \right) \frac{1}{c_F} \biggl(\mathbb{K}^{-1}_{rsmn} \frac{\p \mathbb{K}_{mnij}}{(\p\nabla \lambda)_{pq}} \mathbb{K}^{-1}_{ijkl} \biggr)\bigg|_{(\bar{F}, \nabla \lambda)} \p_q \delta \lambda_p \biggr) \, dx \\
        & \qquad + \int_\Omega \left( - \bar{P}_{ij}\big|_{\bar{F}} \p_j \delta \lambda_i  -  \bar{F}_{ij} \delta \mu_{ij} - \bar{y}_i \p_j \delta \mu_{ij} + b_ i \delta \lambda_i \right) \, dx \quad + \int_{\p \Omega_p} p_i \delta \lambda_i\, da \quad + \int_{\p \Omega_y} y^{(b)}_i \delta \mu_{ij} n_j \, da
    \end{aligned}
\end{equation*}
with the evaluation
\begin{equation*}
    \delta S_Q \big|_{\delta D} [\bar{D}] = \int_\Omega \left( - \bar{P}_{ij}\big|_{\bar{F}} \p_j \delta \lambda_i  -  \bar{F}_{ij} \delta \mu_{ij} - \bar{y}_i \p_j \delta \mu_{ij} + b_i \delta \lambda_i \right) \, dx \quad + \int_{\p \Omega_p} p_i \delta \lambda_i\, da \quad + \int_{\p \Omega_y} y^{(b)}_i \delta \mu_{ij} n_j \, da.
\end{equation*}

Our interest is in examining the second variation of $S_Q$, at the dual state $\bar{D}$ and in arbitrarily fixed directions $\delta D$ and $d D$ (which we will assume to be sufficiently regular so that integrals containing mutual products of them and their derivatives further multiplied by bounded functions, are finite). To do so, it suffices to consider states $D$ with its derivatives to have similar regularity (surely $\bar{D}$ belongs to this class), and then find that
\begin{equation}\label{eq:sec_var}
    \begin{aligned}
       & d\delta S_Q \big|_{(\delta D, d D)} [D] = \\
       & \qquad  - \int_\Omega  \biggl( \frac{1}{c_F} \left(d \mu_{rs} + \p_j d \lambda_i A_{ijrs}\big|_{\bar{F}} \right)  \left( \mathbb{K} \big|_{(\bar{F},\nabla \lambda)}^{-1} \right)_{rskl} \left(\delta \mu_{kl} + \p_j \delta \lambda_i A_{ijkl}\big|_{\bar{F}} \right)  \ + \ \frac{1}{c_y} \p_j d \mu_{ij} \p_k \delta \mu_{ik} \biggr)\, dx \\
       & \qquad + \frac{1}{c_F^2}  \bigl( \cdots \bigr) \quad  + \quad \frac{1}{c_F^3} \bigl( \cdots \bigr),
    \end{aligned}
\end{equation}
where the coefficients of $c_F^{-2}, c_F^{-3}$ are finite (the negative powers on $c_F$ arise from repeated appearances of $\frac{\p \mathbb{K}}{\p \nabla \lambda}$ in the integrands comprising these terms), so that for $c_F \gg 1$ the second variation is controlled by the first term on the right hand side of \eqref{eq:sec_var}. Furthermore, for any dual state for which $\nabla \lambda$ is bounded on the domain, $c_F$ can be chosen large enough so that $\mathbb{K}$ and $\mathbb{K}^{-1}$ are positive-definite. 

Then, for $c_F, c_y \gg 1$, the second variation at such states is \textit{non-positive} and it is this feature of the dual nonlinear elastostatics problem that we term as its \textit{hidden convexity} (with a trivial change of sign). 

This result is also obtained rigorously and without approximation in Sec.~\ref{sec:var_anal}, with a slight reinterpretation of the DtP mapping-determining condition `$\frac{\p \scl}{\p U} = 0$.'

The second variation at the state $\bar{D}$ is given, regardless of the magnitude of $c_F$, by
\begin{equation*}
    d\delta S_Q \big|_{(\delta D, d D)} [\bar{D}] = - \int_\Omega  \biggl( \frac{1}{c_F} \left(d \mu_{rs} + \p_j d \lambda_i A_{ijrs}\big|_{\bar{F}} \right)  \left(\delta \mu_{rs} + \p_j \delta \lambda_i A_{ijrs}\big|_{\bar{F}} \right)  \ + \ \frac{1}{c_y} \p_j d \mu_{ij} \p_k \delta \mu_{ik} \biggr)\, dx
\end{equation*}
and is manifestly semi-definite, \textit{regardless of the definiteness properties of the physical tensor of elastic moduli $A$}.

An important corollary of the above observation is that if the primal base state $(\bar{y}, \bar{F})$ is chosen as a solution to the primal problem - even in a situation where the physical moduli field $x \mapsto A|_{\bar{F}(x)}$ takes positive-definite and negative definite values in distinct parts of the domain so that the \textit{primal problem \eqref{eq:elastostat} is unstable at this state} - a corresponding dual solution, given explicitly by the dual state $\bar{D}$, is neutrally stable to the class of variations considered and, in the sense of the DtP correspondence adopted here, confers a certain stability to the primal solution, when viewed as obtained from the dual variational principle.

\section{Dual nonlinear Elastodynamics}\label{sec:elastodyn}
An appealing feature of the dual methodology is that it seamlessly extends to time-dependent primal problems, leading to well-defined boundary-value-problems in space-time domains. We  consider the system of elastodynamic equations
\begin{subequations}
    \label{eq:elastodyn}
    \begin{align}
        \p_j P_{ij}(F) + b_i - \rho_0 \,\p_t v_i & = 0 \ \mbox{in} \ \Omega \times (0,T) \label{eq:equil}\\
        \p_j y_i - F_{ij} & = 0  \ \mbox{in} \ \Omega \times (0,T) \label{eq:y-F}\\
        \p_t y_i - v_i & = 0  \ \mbox{in} \ \Omega \times (0,T) \label{eq:y-v}\\
        P_{ij}n_j  = p_i  \ \mbox{on} \ \p\Omega_p \times (0,T) \qquad & ; \qquad y_i = y^{(b)}_i  \ \mbox{on} \ \p\Omega_y \times (0,T) \label{eq:bc_dyn}\\
        \left. y_i(x, 0) = y_i^{(0)}\right|_x \quad & ; \quad \left. v_i(x,0) = v_i^{(0)}\right|_x \ \mbox{on} \ \Omega \label{eq:ic_dyn},
    \end{align}
\end{subequations}
where $[0,T]$ is an interval of time, $P$ satisfies the same conditions as in Sec.~\ref{sec:elastostat} and $\rho_0$ is a given time-independent, positive, scalar-valued function on $\Omega$ representing the mass density in the reference configuration.

Following the steps of Sec.~\ref{sec:elastostat}, we define the pre-dual functional
\begin{equation*}
    \begin{aligned}
       &  \widehat{S}_H [v,y,F, \lambda, \gamma, \mu]  = \\
       &  \qquad \int_\Omega \int^T_0 \left( - P_{ij}\big|_F \p_j \lambda_i + \lambda_i b_i + \rho_0 \, v_i \p_t \lambda_i \right)\,dx dt  \quad + \quad \int_\Omega \int^T_0 (- y_i \p_j \mu_{ij} - F_{ij} \mu_{ij})\, dxdt \\
       & \quad +  \int_\Omega \int^T_0 \left( - y_i \p_t \gamma_i - \gamma_i v_i \right) \, dx dt \quad + \quad \int_\Omega \int^T_0  H(v, y, F;x, t) \, dx dt \\
    & \quad + \int_{\p \Omega_p} \int^T_0 p_i \lambda_i \, da  \ + \int_{\p \Omega_y} \int^T_0 y^{(b)}_i \mu_{ij} n_j \, da + \int_\Omega  v^{(0)}_i(x) \lambda_i(x, 0) \, dx - \int_\Omega  y^{(0)}_i(x) \gamma_i(x, 0) \, dx.
    \end{aligned}
\end{equation*}
In the above,  for any boundary term obtained on integration by parts, all primal boundary conditions and initial conditions are imposed. Any integral on a part of the boundary of the space-time region $\Omega \times (0, T)$ where no primal `boundary' data is available is simply ignored at this stage.

We define the boundary of the space-time domain as $\calB := (\p \Omega   \times (0,T)) \cup (\Omega \times \{0,T\})$.

Defining now
\begin{equation}
    \label{eq:defs}
    \begin{aligned}
       &  U := (v, y, F); \qquad D := (\lambda,  \gamma, \mu); \qquad \dee := (\nabla \lambda, \lambda, \p_t \lambda, \nabla \mu, \mu,\p_t \gamma, \gamma );\\
        & \scl_H(U,\dee,x,t) := - P_{ij}\big|_F \,\p_j \lambda_i + b_i \lambda_i + \rho_0 v_i \p_t \lambda_i - y_i \p_j \mu_{ij} - \mu_{ij} F_{ij} - y_i \p_t \gamma_i - v_i \gamma_i + H(v, y,F;x, t),
    \end{aligned}
\end{equation}
and requiring the choice of the potential $H$ to be such that it facilitates the existence of a function
\[
U = U_H(\dee, x, t)
\]
which satisfies
\begin{equation}\label{eq:dLdU_dyn}
\frac{\p \scl_H}{\p U} \left(U_H(\dee, x, t),\dee, x, t \right) = 0 \qquad \forall \ (\dee,x,t),
\end{equation}
it is observed that defining the \textit{dual} functional as
\begin{equation}\label{eq:dual_fnc_dyn}
        \begin{aligned}
        S_H[D] & := \widehat{S}_H \left[ v^{(H)}, y^{(H)}, F^{(H)}, \lambda, \mu, \gamma \right] \\
        & = \int_\Omega \int^T_0 \scl_H\left(U_H(\dee,x), \dee \right) \, dxdt + \int_{\p \Omega_p} \int^T_0 p_i \lambda_i \, da dt \\
       & \qquad  + \int_{\p \Omega_y} \int^T_0 y^{(b)}_i \mu_{ij} n_j \, da dt  + \int_\Omega  v^{(0)}_i(x) \lambda_i(x, 0) \, dx - \int_\Omega  y^{(0)}_i(x) \gamma_i(x, 0) \, dx \\
       & \qquad \mbox{with (arbitrarily) specified $\lambda$ on $\calB \backslash ( \ (\p \Omega_p \times (0,T)) \cup (\Omega \times \{0\}) \ )$ and}\\
       & \qquad \gamma \ \mbox{on} \ \calB \backslash ( \ (\p \Omega_y \times (0,T)) \cup (\Omega \times \{0\}) \ ),
    \end{aligned}    
\end{equation}
and noting \eqref{eq:dLdU_dyn}, the first variation of $S_H$ (about a state $D$ in the direction $\delta D$, with $\delta \lambda$ and $\delta \gamma$ constrained to vanish on parts of the boundary of $\calB$ on which $\lambda$ and $\gamma$ are specified, respectively), is given by
\begin{equation*}
    \begin{aligned}
        \delta S_H & = \bigintsss_\Omega \bigintsss_0^T \frac{ \p \scl_H}{\p \dee} \left(U_H(\dee,x), \dee, x,t \right) \cdot \delta \dee \, dx dt + \int_{\p \Omega_p} \int^T_0 p_i \delta \lambda_i \, da dt  \\
& \quad \ + \int_{\p \Omega_y} \int^T_0 y^{(b)}_i \delta \mu_{ij} n_j \, da dt + \int_\Omega  v^{(0)}_i(x) \delta \lambda_i(x, 0) \, dx - \int_\Omega  y^{(0)}_i(x) \delta \gamma_i(x, 0) \, dx.\\
    \end{aligned}
\end{equation*}
For the same reasons as in Sec.~\ref{sec:elastostat}, it can again be checked that the Euler-Lagrange equations and natural boundary conditions (on the space-time domain $\calB$) - of the dual functional \eqref{eq:dual_fnc_dyn} are exactly the system \eqref{eq:elastodyn} with $U$ substituted by $U_H(\dee,x,t)$.

As argued in \cite[Sec.~7]{action_2} and demonstrated in \cite{KA1}, the imposition of final-time conditions on the dual fields does not in any way constrain the dual problem from generating solutions to the primal \textit{initial}-boundary-value-problem \eqref{eq:elastodyn} through the DtP mapping.

We now develop the explicit form of the dual \textit{action} for nonlinear elastodynamics for the quadratically nonlinear stress:
\begin{equation}\label{eq:quad_strs_dyn}
\begin{aligned}
     P_{ij}(F(x,t)) & = \bar{P}_{ij}\left(\bar{F}(x,t) \right) + A_{ijkl}|_{\bar{F}(x,t)} \left( F_{kl}(x,t) - \bar{F}_{kl}(x,t)\right) \\
     & \quad + \frac{1}{2} B_{ijklmn}|_{\bar{F}(x,t)} \left( F_{kl}(x,t) - \bar{F}_{kl}(x,t)\right)\left( F_{mn}(x,t) - \bar{F}_{mn}(x,t)\right),
\end{aligned}
\end{equation}
We make an arbitrary choice of a base state
\[
(\ (x,t) \mapsto \bar{v}(x,t) = \p_t \bar{y}(x,t),\  (x,t) \mapsto \bar{y}(x,t), \ (x,t) \mapsto \bar{F}(x,t) = \nabla \bar{y} (x,t)\ )
\]
and adopt the following shifted quadratic for the potential $H$:
\begin{equation}\label{eq:Q_dyn}
Q(v, y,F;x,t) = \frac{1}{2} \left( c_v |v - \bar{v}(x,t)|^2 + c_y \left|y - \bar{y} (x,t) \right|^2 + c_F \left|F - \bar{F}(x,t)\right|^2  \right).
\end{equation}

Next, we form the Lagrangian corresponding to the system \eqref{eq:elastodyn} and the potential \eqref{eq:Q_dyn}: 
\begin{equation}\label{eq:Lagrangian_dyn}
    \begin{aligned}
         \scl_Q(v,y,F, \nabla \lambda, \p_t \lambda, \nabla \mu, \mu, \p_t \gamma, \gamma) & = 
        - P_{ij}\big|_F \,\p_j \lambda_i + b_i \lambda_i + \rho_0 v_i \p_t \lambda_i - y_i \p_j \mu_{ij} - \mu_{ij} F_{ij} - y_i \p_t \gamma_i - v_i \gamma_i \\
      & \quad + \frac{1}{2} \left( c_v |v - \bar{v}(x,t)|^2 + c_y \left|y - \bar{y}(x,t) \right|^2 + c_F \left|F - \bar{F}(x,t)\right|^2  \right).
    \end{aligned}
\end{equation}
The DtP mapping corresponding to \eqref{eq:Lagrangian_dyn}- \eqref{eq:quad_strs_dyn}-\eqref{eq:Q_dyn} is then given by the relations
\begin{equation*}
    \begin{aligned}
        c_v\left( v_i^{(Q)} - \bar{v}_i\right) & = \gamma_i - \rho_0 \, \p_t \lambda_i    \\
        c_y \left( y_i^{(Q)} - \bar{y}_i\right) & =  \p_j \mu_{ij} + \p_t \gamma_i \\
        c_F \, \mathbb{K}_{rskl}\big|_{(\bar{F}, \nabla \lambda)} \left( F_{kl}^{(Q)} - \bar{F}_{kl} \right) & = \mu_{rs} +  \p_j \lambda_i A_{ijrs}|_{\bar{F}},
    \end{aligned}
\end{equation*}
where $\mathbb{K}$ is defined in \eqref{eq:K}.
Utilizing these ingredients, the dual action,
$$S_Q[D] := \widehat{S}_Q \left[ v^{(Q)}, y^{(Q)}, F^{(Q)}, \lambda, \mu, \gamma \right]
$$ 
corresponding to `quadratically' nonlinear elastodynamics \eqref{eq:elastodyn} and \eqref{eq:Q_dyn} is given by (again dropping the superscript $^{(Q)}$ for brevity)
\begin{equation*}
    \begin{aligned}
        S_Q[D] & = - \int_\Omega \int^T_0  \bar{P}_{ij}\big|_{\bar{F}} \p_j \lambda_i  + \bar{v}_i \left( \gamma_i - \rho_0 \, \p_t \lambda_i \right) + \bar{y}_i \left( \p_j \mu_{ij} + \p_t \gamma_i \right) + b_i \lambda_i \, dx dt \\
        & \quad - \frac{1}{2} \int_\Omega \int^T_0 \frac{1}{c_v} \left( \gamma_i - \rho_0 \, \p_t \lambda_i \right) \left( \gamma_i - \rho_0 \, \p_t \lambda_i \right) \  + \  \frac{1}{c_y} \left( \p_j \mu_{ij} + \p_t \gamma_i \right) \left( \p_j \mu_{ij} + \p_t \gamma_i \right) \, dx dt\\
        & \quad - \frac{1}{2} \int_\Omega \int^T_0  \Big( \mu_{rs} + \p_j \lambda_i A_{ijrs}\big|_{\bar{F}} \Big) \frac{1}{c_F} \left( \mathbb{K} \big|_{(\bar{F},\nabla \lambda)}^{-1} \right)_{rskl} \Big( \mu_{kl} + \p_j \lambda_i A_{ijkl}\big|_{\bar{F}} \Big) \, dx dt\\
        & \quad + \int_{\p \Omega_p} \int^T_0 p_i \lambda_i \, da dt  \ + \int_{\p \Omega_y} \int^T_0 y^{(b)}_i \mu_{ij} n_j \, da dt  \\
        & \quad + \int_\Omega  v^{(0)}_i(x) \lambda_i(x, 0) \, dx - \int_\Omega  y^{(0)}_i(x) \gamma_i(x, 0), \, dx.
    \end{aligned}
\end{equation*}
\section{Local degenerate ellipticity of dual nonlinear Elasticity - statics and dynamics} \label{sec:deg_ell}
This treatment is a special case of that in \cite[Sec.~3]{dual_cont_mech_plas} and is included here for a self-contained account.

It is easier to see the following argument for a general setup, which is what we employ. For this section, let Greek lower-case indices belong to the set $\{0,1,2,3\}$ representing Rectangular Cartesian space-time coordinates; $0$ represents the time coordinate when the PDE is time-dependent. Let upper-case Latin indices belong to the set $\{1,2,3, \cdots,N \}$, indexing the components of the $N \times 1$ array of primal variables, $U$, with, possibly, a conversion to first-order form as necessary. Now consider the system of primal PDE
\begin{equation}\label{eq:conserv_source}
    \p_\alpha (\calF_{\Gamma\alpha}(U)) + G_\Gamma(U) = 0, \qquad \Gamma = 1, \ldots, N^*
\end{equation}
where upper-case Greek indices index the number of equations involved, after conversion to first-order form when needed.  
For example, consider the three conservation laws for nonlinear elastodynamics \eqref{eq:equil}. The conversion to first-order form (so that the Lagrangian $\scl$ \eqref{eq:L} below contains no derivatives in the primal variables) requires the addition of nine more primal variables $F$ and three more as $v$ \eqref{eq:y-F}, \eqref{eq:y-v}. These additional twelve relations can be written in the form
\begin{equation}\label{eq:first-order}
 \mathcal{A}_{\Gamma I \alpha} \p_\alpha U_I - \mathcal{B}_{\Gamma I} U_I = 0, \qquad \Gamma = 4,\ldots 15,   
\end{equation}
where $\mathcal{A}, \mathcal{B}$ are constant matrices (with $\mathcal{B}$ diagonal in many cases) that define the augmentation of the primal list from $(y)$ to $(y, F, v)$, and define the augmenting primal variables as, in general, linear combinations of the partial derivatives of components of $U$. The equation set \eqref{eq:first-order} can be expressed in  the form  \eqref{eq:conserv_source}.

We assume that the functions $U \mapsto \frac{\p^2 \calF_\Gamma}{\p U_P \p U_R}(U)$ and $U \mapsto \frac{\p^2 G_\Gamma}{\p U_P \p U_R}(U)$ are bounded functions on their domains.

Let $D$ be the $N^* \times 1$ array of dual fields and, as earlier, let us consider a shifted quadratic for the potential $H$, characterized by a diagonal matrix $[a_{kj}]$ with constant positive diagonal entries so that the Lagrangian takes the form
\begin{equation}\label{eq:L}
    \scl(U,D,\nabla D, \bar{U}) := - \calF_{\Gamma \alpha}(U) \p_\alpha D_\Gamma + D_\Gamma G_\Gamma(U) + \half (U_k - \bar{U}_k) a_{kj} (U_j - \bar{U}_j).
\end{equation}
Then the corresponding DtP mapping, obtained by `solving $\frac{\p \scl}{\p U} = 0$ for $U$ in terms of $(\nabla D, D, \bar{U})$,' is given by the implicit equation
\begin{equation}\label{eq:deg_ell_map}
    U_J^{(Q)}(\nabla D, D, \bar{U}) = \bar{U}_J + (a^{-1})_{JK} \left( \frac{\p \calF_{\Gamma \alpha}}{\p U_K} \bigg|_{U^{(Q)}(\nabla D, D, \bar{U})} \p_\alpha D_\Gamma - D_\Gamma \frac{\p G_\Gamma}{\p U_K}\bigg|_{U^{(Q)}(\nabla D, D, \bar{U})} \right).
\end{equation}
By the fundamental property of the dual scheme, the dual E-L equation is then given by
\begin{equation}\label{eq:dual_EL}
    \p_\alpha \Big( \calF_{\Gamma \alpha} \big( U(\nabla D, D, \bar{U})) \Big) + G_\Gamma ( U(\nabla D, D, \bar{U}) ) = 0
\end{equation}
(where we have dropped the superscript $^{(Q)}$ for notational convenience), whose ellipticity is governed by the term
\[
\mathbb{A}_{\Gamma \alpha \Pi \mu}(\nabla D, D, \bar{U}) := \frac{\p \calF_{\Gamma \alpha}}{\p U_P}\bigg|_{U(\nabla D, D, \bar{U})} \, \frac{\p U_P}{\p (\nabla D)_{\Pi \mu}}\bigg|_{U(\nabla D, D, \bar{U})}.
\]
From \eqref{eq:deg_ell_map} we have
\begin{equation*}
    \begin{aligned}
         a^{-1}_{PR} \left(  \delta_{\Gamma \Pi} \delta_{\mu \alpha} \frac{ \p \calF_{\Gamma \alpha}}{\p U_R} \, + \, \p_\alpha D_\Gamma \frac{ \p^2 \calF_{\Gamma \alpha}}{\p U_R \p U_S} \frac{\p U_S}{\p (\nabla D)_{\Pi \mu}}  - D_\Gamma \frac{ \p^2 G_\Gamma}{\p U_R \p U_S} \frac{\p U_S}{\p (\nabla D)_{\Pi \mu}} \right)     & =  \frac{\p U_P}{\p (\nabla D)_{\Pi \mu}} \\
        \Longrightarrow \left( \delta_{PS} - a^{-1}_{PR} \p_\alpha D_\Gamma \frac{ \p^2 \calF_{\Gamma \alpha}}{\p U_R \p U_S} + a^{-1}_{PR} D_\Gamma \frac{ \p^2 G_\Gamma}{\p U_R \p U_S} \right) \frac{\p U_S}{\p (\nabla D)_{\Pi \mu}} & = a^{-1}_{PR} \frac{\p \calF_{\Pi \mu}}{\p U_R},
    \end{aligned}
\end{equation*}
and so
\begin{equation*}
    \mathbb{A}_{\Gamma \alpha \Pi \mu}(0,0,\bar{U}) = \frac{\p \calF_{\Gamma \alpha}}{\p U_P}\bigg|_{\bar{U}} \, a^{-1}_{PR} \, \frac{\p \calF_{\Pi \mu}}{\p U_R}\bigg|_{\bar{U}},
\end{equation*}
which is \textit{positive semi-definite} on the space of $N^* \times 3$ (or $N^* \times 4$) matrices. This establishes the degenerate ellipticity of the dual system at the state $x \mapsto D(x) = 0$.

To examine the ellipticity-related properties of the system in a bounded neighborhood, say $\mathcal{N}$, of $(D = 0, \nabla D = 0) \in \mathbb{R}^{N^*} \times \mathbb{R}^{N^* \times \bar{\alpha}}, \bar{\alpha} = 3,4$, we define
\[
M_{PS} := \delta_{PS} - a^{-1}_{PR} \left( \p_\alpha D_\Gamma \frac{\p^2 \calF_{\Gamma \alpha}}{\p U_R \p U_S} + D_\Gamma  \frac{\p^2 G_\Gamma}{\p U_R \p U_S} \right),
\]
and note that
\[
\frac{\p U_P}{\p (\nabla D)_{\Pi \mu} } = M^{-1}_{PQ} a^{-1}_{QR} \frac{\p \calF_{\Pi \mu}}{\p U_R},
\]
where $M^{-1}$ exists and is positive definite by the boundedness of $\mathcal{N}$ and the second derivatives of the functions $\calF$ and $G$, along with an appropriately large choice of the elements of the diagonal matrix $[a_{ij}]$. 

The degenerate ellipticity or `convexity' of the system \eqref{eq:conserv_source} in the neighborhood $\mathcal{N}$ is now defined as the positive semi-definiteness of the matrix $\mathbb{A}$  on the space $\mathbb{R}^{N^* \times \bar{\alpha}}$ of matrices, and this in turn is governed by the matrix
\[
\mathbb{A}^{(sym)}_{\Gamma \alpha \Pi \mu}\bigg|_{(\nabla D, D, \bar{U})} = \frac{\p \calF_{\Gamma \alpha}}{\p U_P}\bigg|_{U(\nabla D, D, \bar{U})} \frac{1}{2} \left( M^{-1}_{PQ}\bigg|_{U(\nabla D, D, \bar{U})}  a^{-1}_{QR} + M^{-1}_{RQ}\bigg|_{U(\nabla D, D, \bar{U})}  a^{-1}_{QP}\right)\frac{\p \calF_{\Pi \mu}}{\p U_R}\bigg|_{U(\nabla D, D, \bar{U})}.
\]
By the positive definiteness of the matrix $[M_{PS}]$ in the neighborhood $\mathcal{N}$, it follows that
\[
\xi^{\Gamma \alpha} \ \mathbb{A}^{(sym)}_{\Gamma \alpha \Pi \mu}\bigg|_{(\nabla D, D, \bar{U})} \xi^{\Pi \mu}  \geq 0 \qquad \forall \qquad (D, \nabla D) \in \mathcal{N}, \quad \xi \in \R^{N^* \times \bar{\alpha}}
\]
which establishes a `local' degenerate ellipticity of the system \eqref{eq:conserv_source}. We note that degenerate ellipticity is stronger than the Legendre-Hadamard condition given by the requirement of positive semi-definiteness of $\mathbb{A}$ on the space of tensor products from $\mathbb{R}^{N^*} \otimes \mathbb{R}^{\bar{\alpha}}$, and not directly comparable to the strong-ellipticity condition, since it is weaker than the latter when restricted to the space $\mathbb{R}^{N^*} \otimes \mathbb{R}^{\bar{\alpha}}$ but simultaneously requiring semi-definiteness on the larger space of $\mathbb{R}^{N^* \times \bar{\alpha}}$. Also of note is that degenerate ellipticity does not preclude the failure of ellipticity characterized by the condition $\det [\mathbb{A}_{\Gamma \alpha \Pi \mu} n_\alpha n_\mu] \neq 0$ for all unit direction $n \in \R^ {\bar{\alpha}}$, $\bar{\alpha} = 3$ or $4$, thus allowing for weak (gradient) discontinuities of weak solutions $x \mapsto \dot{D}(x)$ of the linearization of \eqref{eq:dual_EL}.

If a solution of the primal system is close to the base state $\bar{U}$, then it seems natural to  expect, due to this local degenerate ellipticity, that such a solution can be obtained in a `stable' manner by the dual formulation \textit{designed by the choice of the auxiliary potential $H$ as a shifted quadratic about the base state $\bar{U}$}, for instance by an iterative scheme starting from a guess $(D = 0, U = \bar{U})$.

Our experience (Sec.~\ref{sec:compute} of this work, \cite{KA1, KA2, a_arora_thesis}) shows that this observation is of great practical relevance in using the dual scheme, and we consistently exploit it in all our computational approximations.

\section{Variational Analysis of dual nonlinear Elastostatics}\label{sec:var_anal}

\textit{In the remaining sections of the paper we will assume that the body force $b = 0$}.

In this section, we will follow the approach outlined in Section \ref{sec:elastostat} and study the corresponding dual functional $S_H$.
First, note that under sufficient growth and regularity assumptions on $H$ the relation \eqref{eq:dLdU} is satisfied if $U_H(\dee,x)$ maximizes $-\mathcal{L}_H(\cdot, \dee, x)$ for fixed values of $(\dee,x)$. 
If for all $D$ the maximizer is unique then we call the mapping $(\dee(x),x) \mapsto U_H(\dee(x),x)$ the dual-to-primal mapping. 
With a little abuse of notation we also use the same notation $U_H$ for the mapping between functions defined for a pair of functions $D=(\lambda,\mu)$ as $D \to ( x \to U_H(\dee(x),x))$. 
In this setting the dual functional can equivalently be written as (c.f.~Proposition \ref{prop: SH = I}, \eqref{eq: def I neo hook} and \eqref{eq: def I neo hook2} for the precise statement for the examples considered below) 
\begin{equation} \label{eq: tildeSH}
\tilde{S}_H[D]:= \sup_{y,F} - \widehat{S}_H[y,F,\lambda,\mu] = -S_H[D], 
\end{equation}
but the definition of $\tilde{S}_H$ makes sense more generally even in the absence of a well-defined mapping $U_H$, which suffices for the notion of a variational dual solution defined below. It is this definition of $\tilde{S}_H$ that we adopt in this section.
We note that the functional $\hat{S}_H$ is affine in the functions $\lambda$, $\nabla \lambda$, $\operatorname{div}(\mu)$ and $\mu$.
In particular, if formulated in appropriate spaces it will be continuous with respect to weak convergence of these quantities. 
This implies that $\tilde{S}_H$ as the $\sup$ over these weakly continuous functions is weakly lower semicontinuous.
This renders the dual problem $\tilde{S}_H$ to be rather well-suited for applying the direct method of the Calculus of Variations to prove the existence of a minimizer. 
If the dual functional $\tilde{S}_H$ is regular enough, such a minimizer would then satisfy the Euler-Lagrange equations of $S_H$ in a weak form (which are the same as for $\tilde{S}_H$) which correspond to the equations \eqref{eq:elastostat} if there is a well-defined mapping from the dual states to the primal states as explained above.
The direct method of the Calculus of Variations essentially corresponds to performing the following three steps:
\begin{enumerate}
    \item Show that $\inf_{\lambda,\mu} \tilde{S}_H > -\infty$ and take a minimizing sequence $(\lambda_k, \mu_k)$ such that $\tilde{S}_H(\lambda_k,\mu_k) \to \inf_{\lambda,\mu} \tilde{S}$.
    \item Prove up to a non-relabeled subsequence (weak) compactness for the minimizing sequence $(\lambda_k,\mu_k) \rightharpoonup (\lambda,\mu)$.
    \item Show that $S_H$ is sequentially lower-semicontinuous with respect to the convergence established in 2.~to conclude that $\tilde{S}_H(\lambda, \mu) \leq \liminf_{k \to \infty} \tilde{S}_H(\lambda_k,\mu_k) = \inf_{\lambda,\mu} \tilde{S}_H$ i.e., $(\lambda,\mu)$ minimizes $\tilde{S}_H$.
\end{enumerate}
Note that by the argumentation above the dual functional $\tilde{S}_H$ is (at least formally) lower semi-continuous with respect to weak convergence. 
Consequently, the main work in this setting is to establish step 2.

Let us also note that as the function $\hat{S}_H$ is affine in $\lambda$ and $\mu$ the function $\tilde{S}_H$ as the supremum over affine functions is automatically convex (however, most likely not strictly convex). 
This means that there is a large toolbox for the computational task to find a minimizer of the dual functional $\tilde{S}_H$.\\

In light of the discussion above, we define the notion of a \textit{variational dual solution} to the equation \eqref{eq:elastostat} and a \textit{dual solution}.
\begin{definition} \label{def: dual sol} Given a Borel-measurable function $H$ we say that a pair $(\lambda,\mu)$ is a variational dual solution to the equation \eqref{eq:elastostat} if it minimizes the dual functional $\tilde{S}_H$. 
Moreover, we say that a pair $(\lambda,\mu)$ is a dual solution to \eqref{eq:elastostat} if the dual functional $\tilde{S}_H$ is differentiable at $(\lambda,\mu)$, the pair $(\lambda,\mu)$ satisfies the weak form of the corresponding Euler-Lagrange equation for $\tilde{S}_H$ and there exists a smooth dual-to-primal mapping $U_H$ corresponding to $\tilde{S}_H$.
\end{definition}

By the considerations above, a dual solution gives rise to a weak solution of \eqref{eq:elastostat}. 
Note that for the notion of a variational dual solution in principle we do not require $H$ to satisfy any additional assumption. 
However, the notion of a dual variational solution is only  meaningful if the function $H$ is chosen in a way such that the dual functional $\tilde{S}_H$ allows for nontrivial minimizers.
For example, for $H=0$ it is easy to check that for the examples discussed below it follows $\tilde{S}_H[D] = +\infty$ except for $D = (0,0)$, where $\tilde{S}_H[D] = 0$. 
In particular, the dual functional is minimized at $(0,0)$ but every pair of primal variables is a maximizer for the middle term of \eqref{eq: tildeSH}.
This means that not every variational dual solution corresponds to a dual solution in the sense of Definition \ref{def: dual sol} or gives rise to a weak solution of \eqref{eq:elastostat}.
Consequently, ideally we would like to find a function $H$ such that the dual functional $\tilde{S}_H$ is regular and the dual-to-primal mapping is well-defined and regular.
In order to obtain these two properties, it appears to be reasonable (but in most cases not sufficient) to choose $H$ such that $H$ has a strictly positive definite Hessian at all points (in particular $H$ is convex) and the growth of $H$ at $+\infty$ is not smaller than the growth of the nonlinearity $P(F)$ so that a maximizer for the problem in the middle of \eqref{eq: tildeSH} exists and might be unique (at least for a  large number of dual configurations). 

In the case of a hyperelastic materials the typical variational strategy to obtain a weak solution to \eqref{eq:elastostat} would be to prove existence of a minimizer for its free energy functional subject to $y = y^{(b)}$ on $\Omega_y$. Much research in this direction was inspired by the seminal works on functionals with quasiconvex  \cite{Morrey52} or polyconvex energy densities \cite{Ball77,Muller90}.
However, these notions of convexity do not hold for numerous classical models, e.g.~the Saint Venant - Kirchhoff model discussed in the section below. 
In addition, even if existence of a minimizer for the free energy can be established, it is not clear that the minimizer satisfies the weak formulation of the corresponding Euler-Lagrange equation, see the discussion in \cite{Ball02} and the references therein (in the case of a convex integrand see, for example, \cite{KrKo22,Kr15} for a positive result).

Hence, we propose the concept of a variational dual solution as another variational solution concept for nonlinear elastostatics given by \eqref{eq:elastostat}. 

In the following we will show that the concept of variational dual solutions is applicable in a meaningful way to some standard models, namely, we will prove, for auxiliary potentials $H$ that render the dual functional $\tilde{S}_H$ non-trivial, existence of variational dual solutions to \eqref{eq:elastostat} for the Saint Venant-Kirchhoff model (in all dimensions) and the incompressible Neo-Hookean model (in two dimensions). 
We will focus on proving coercivity for the dual functionals $\tilde{S}_H$ which guarantees that step 2.~in the direct method can be performed.
In these examples the formal argument for the lower semi-continuity of the dual functional $\tilde{S}_H$ can also be made rigorous and hence existence of minimizers of $\tilde{S}_H$ can be established rigorously.

Following the arguments in the beginning of Section \ref{sec:elastodyn} one can set up analogously a dual variational problem in the elastodynamic setting. The dual variational problem is again (formally) lower semi-continuous and convex. Understanding the coercivity of the dual variational problem in the dynamical setting, which would in turn give rise to the existence of the analogue of dual variational solutions, is currently work in progress.

In the following we will assume that $\Omega \subseteq \R^d$ is open, connected and bounded with Lipschitz boundary.

\subsection{The Saint Venant-Kirchhoff model} 

We consider the Saint Venant-Kirchhoff energy
\[
W(F) = \mathsf{G} |E|^2 + \frac{\mathsf{L}}2 \operatorname{tr}(E)^2, 
\]
where $C = F^TF$ is the right Cauchy-Green tensor, $E = \half (C - Id)$ is the Green-Lagrange strain tensor, the shear modulus $\mathsf{G}>0$ and $\mathsf{G} + \frac{d}{2} \mathsf{L} >0$.
Correspondingly, we set
\[
P(F) = DW(F) = F \left( \mathsf{G} \,  F^T F + \frac{\mathsf{L}}2 |F|^2 Id - (\mathsf{G} + \frac{\mathsf{L} d}2) Id \right).
\]

For $y^{(b)} \in W^{1,4}(\Omega;\R^d)$ and $p^{(b)} \in L^{4/3}(\partial \Omega;\R^d)$ we look for weak solutions to the equations \eqref{eq:elastostat} which take the form
\begin{equation} \label{eq: St Venant weak} 
\begin{cases}
&F = \nabla y \text{ in } \Omega, \\
&\operatorname{div} P(F) = 0 \text{ in } \Omega, \\
&P(F) \cdot n = p^{(b)} \text{ on } \partial \Omega_p, \\
&y = y^{(b)} \text{ on } \partial \Omega_y.
\end{cases}
\end{equation}
Precisely, we look for $y \in W^{1,4}(\Omega; \R^{d})$ and $F \in L^{4}(\Omega; \R^{d\times d})$ such that for all $\lambda \in W^{1,4}(\Omega; \R^d)$ with $\lambda= 0$ on $\partial \Omega_y$ and $\mu\in L^{4/3}(\Omega; \R^{d \times d})$ with $\operatorname{div}(\mu) \in L^{4/3}(\Omega;\R^d)$ and $\mu n = 0$ on $\partial \Omega_p$ that
\[
\int_{\Omega} \operatorname{div }(\mu) \cdot y + \mu:F + P(F) : \nabla \lambda\, dx =  \int_{\partial \Omega_y}  y^{(b)} \cdot \mu n \, d\mathcal{H}^{d-1} + \int_{\partial \Omega_p} \lambda\cdot p^{(b)} \, d\mathcal{H}^{d-1}.
\]
Note that the boundary terms are well-defined but have to be understood in the sense of traces i.e., for $p \in [1,\infty]$ there are continuous trace operators $W^{1,p}(\Omega;\R^d) \to W^{1-1/p,p}(\partial \Omega;\R^d) \subseteq L^p(\partial \Omega)$ and $\{ \mu \in L^p(\Omega;\R^d): \operatorname{div}(\mu) \in L^p\} \to \left(W^{1-1/p}(\partial \Omega;\R^d) \right)'$, $\mu \to \mu  n$, see \cite[Theorem 18.40]{LeoniSobolev} and \cite[Theorem 1, page 204]{Lions} for the case $p=2$ (the general case follows with the same proof). In particular, the boundary terms are continuous with respect to weak convergence.

Next, let us mention that the stored energy density $W$ is not rank-$1$ convex and hence neither polyconvex nor quasiconvex (see \cite{Raoult}). In particular, the corresponding stored elastic energy is not weakly lower semi-continuous and existence for \eqref{eq: St Venant weak} via the minimization of the corresponding energy cannot be obtained by the direct method of the Calculus of Variations, c.f.~the discussion in the preamble of this section.

Following the approach explained above we define for 
\[
H(y,F) = \frac14 |y|^4 + \frac{\mathsf{G} + \mathsf{L} /2}{2} |F^T F|^2
\]
the function $g: \R^{d\times d} \times \R^d \times \R^{d\times d} \to \R$ given by
\[
g(A,a,B) = \sup_{F \in \R^{d \times d}, y \in \R^d} a \cdot y + A: F + B: P(F) - H(y,F)
\]
and consider the functional $I: \{ (\mu,\lambda) \in L^{4/3}(\Omega;\R^{d \times d}) \times W^{1,4}(\Omega;\R^d): \operatorname{div}(\mu) \in L^{4/3}(\Omega;\R^d), \mu n = 0 \text{ on } \partial \Omega_p  \text{ and } \lambda= 0 \text{ on } \partial \Omega_y \} \to \R$ defined as 
\begin{equation} \label{eq: def I}
I(\mu,\lambda) = \int_{\Omega} g(\mu,\operatorname{div}(\mu),\nabla \lambda) \, dx - \int_{\partial \Omega_y}  y^{(b)}\cdot \mu n \, d\mathcal{H}^{d-1} - \int_{\partial \Omega_p} \lambda\cdot p^{(b)} \, d\mathcal{H}^{d-1}.
\end{equation}

We will first show that $I$ agrees with the dual functional $\tilde{S}_H$.

\begin{proposition}\label{prop: SH = I}
    It holds for all $\mu \in L^{4/3}(\Omega;\R^{d \times d})$, $\lambda \in W^{1,4}(\Omega;\R^{d\times d})$ that
    \[
\tilde{S}_H[\lambda, \mu] = \int_{\Omega} g(\mu, \operatorname{div}(\mu), \nabla \lambda) \, dx - \int_{\partial \Omega_y}  y^{(b)}\cdot \mu n \, d\mathcal{H}^{d-1} - \int_{\partial \Omega_p} \lambda\cdot p^{(b)} \, d\mathcal{H}^{d-1}.
    \]
\end{proposition}
\begin{proof}
It suffices to show that for $\mu \in L^{4/3}(\Omega;\R^{d \times d})$, $\lambda \in W^{1,4}(\Omega;\R^{d\times d})$ that it holds
\begin{equation}\label{eq: eq SH g}
\sup_{y \in L^{4}(\Omega;\R^d), F \in L^4(\Omega;\R^{d\times d})} \int_{\Omega} \operatorname{div}(\mu) \cdot y + \mu : F + P(F): \nabla \lambda  - H(y,F) \, dx = \int_{\Omega} g(\mu,\operatorname{div}(\mu),\nabla \lambda) \, dx.
\end{equation}
Clearly, it holds that
\[
\sup_{y \in L^{4}(\Omega;\R^d), F \in L^4(\Omega;\R^{d\times d})} \int_{\Omega} \operatorname{div}(\mu) \cdot y + \mu : F + P(F): \nabla \lambda  - H(y,F) \, dx \leq \int_{\Omega} g(\mu,\operatorname{div}(\mu),\nabla \lambda) \, dx.
\]
We will now prove the reverse inequality.
Let $\mathcal{O} = \{ O \subseteq \Omega: O \text{ open} \}$ and consider the countable family of functions 
\begin{align}
\mathcal{A} = \{ (y,F): &F= \sum_{n=1}^N q_n \chi_{E_n} \text{ and } y = \sum_{n=1}^N q_n' \chi_{E_n}, \text{ where }  q_n \in \mathbb{Q}^{d\times d}, \, q_n' \in\mathbb{Q}^d \text{ and } \\ &E_n \text{ is a cuboid with rational side lengths and corners in } \mathbb{Q}^d  \}.
\end{align}
Note that $\mathcal{A}$ is countable. Moreover, it holds for all $(a,A,B) \in \R^d \times \R^{d\times d} \times \R^{d\times d}$ and $x \in \Omega$
\[
\sup_{y,F \in \mathcal{D}} a \cdot y(x) + A:F(x) +  P(F(x)):B - H(y(x),F(x)) = g(A,a,B). 
\]
For $O \in \mathcal{O}$ let us now define $\mathcal{F}(O) = \sup_{y \in L^4, F \in L^4} \int_{O} y\cdot \operatorname{div}(\mu) + \mu:F + P(F):\nabla \lambda - H(y,F) \, dx$.
It can be shown that for $O,O' \in \mathcal{O}$ such that $\operatorname{dist}(O,O') > 0$ it follows that 
\[
\mathcal{F}(O \cup O') = \mathcal{F}(O) + \mathcal{F}(O').
\]
Moreover, it holds by definition for each $(y,F) \in \mathcal{A}$ and $O \in \mathcal{O}$ that
\[
\mathcal{F}(O) \geq \int_{O} y\cdot \operatorname{div}(\mu) + \mu:F + P(F):\nabla \lambda - H(y,F) \, dx.  
\]
Then it follows from \cite[Lemma 15.2]{braides2002gamma} that
\[
\mathcal{F}(\Omega) \geq \int_{\Omega} \sup_{(y,F) \in \mathcal{A}} \left[y\cdot \operatorname{div}(\mu) + \mu:F + P(F):\nabla \lambda - H(y,F) \right] \, dx = \int_{\Omega} g(\mu,\operatorname{div}(\mu),\nabla \lambda) \, dx.
\]
This proves the reverse inequality.
\end{proof}

As explained above, in order to prove existence of minimizers for $\tilde{S}_H$ it suffices by the direct method of the Calculus of Variations to prove that $\tilde{S}_H$ is coercive. Hence, we will now estimate the function $g$ from below.

\begin{proposition}\label{prop: lb g}
    There exists $c>0$ such that it holds for all $(a,A,B) \in \R^d \times \R^{d\times d} \times \R^{d \times d}$
    \[
    g(A,a,B) \geq c \left( |a|^{4/3} + |A|^{4/3} + |B |^4 \right) - \frac1c.
    \]
\end{proposition}
\begin{proof}
By optimizing in $y$ it follows that
\begin{align*}
&g(A, a, B) \\
= &\frac34 |a|^{4/3} + \sup_F \left[ \left( A- (\mathsf{G} + \frac{\mathsf{L} d}{2}) B\right) : F + B: \left(\mathsf{G} \, F F^T F + \frac{\mathsf{L}}2 |F|^2 F \right) - \frac{\mathsf{G} + \mathsf{L} /2}{2} |F^T F|^2 \right] \\
=  &\frac34 |a|^{4/3} + 2(\mathsf{G} + \mathsf{L}/2) \sup_F \left[ \tilde{A} : F + \frac{1}{2(\mathsf{G} + \mathsf{L}/2)} B: \left(\mathsf{G} \, F F^T F + \frac{\mathsf{L}}2 |F|^2 F \right) - \frac14 |F^T F|^2 \right] \\
=: &\frac34 |a|^{4/3} + 2(\mathsf{G} + \mathsf{L}/2) h(\tilde{A}, B ).
\end{align*}
where $\tilde{A} = \frac{1}{2\mathsf{G} + \mathsf{L}}\left( A- (\mathsf{G} + \frac{\mathsf{L} d}{2}) B\right)$.
Next, we estimate $h$ from below.
\begin{enumerate}
    \item If $|B |^3 \leq \frac{27}{8^3} |\tilde{A}|$ we find for $F = |\tilde{A}|^{-2/3} \tilde{A}$ that
    \begin{align*}
    h(\tilde{A}, B ) &\geq  |\tilde{A}|^{4/3} - |\tilde{A}| |B | - \frac14 |F|^4 \\
    &= \frac34 |\tilde{A}|^{4/3} - |\tilde{A}| |B | \geq \frac38 |\tilde{A}|^{4/3} \geq \frac3{16} ( |\tilde{A}|^{4/3} + |B |^4 ),
    \end{align*}
    where we used that $|F^TF|^2 \leq |F|^4$.
    \item If $\frac{27}{8^3} |\tilde{A}| < |B |^3 \leq \frac{8d}{9} |\tilde{A}|$ we find for $F = \frac1{\sqrt{3}} |B |^{-1/2} |\tilde{A}|^{-1/2} \tilde{A}$ that
    \begin{align*}
        h(\tilde{A},B ) &\geq \frac1{\sqrt{3}} |\tilde{A}|^{3/2} |B |^{-1/2} - \frac1{3 \sqrt{3}} |\tilde{A}|^{3/2} |B |^{-1/2} - \frac1{36} |\tilde{A}|^2 |B |^{-2} \\
        &\geq \left( \frac2{3 \sqrt{3}} - \frac1{36} \sqrt{\frac{8^3}{27}}  \right) |\tilde{A}|^{3/2} |B |^{-1/2} \\
        &\geq \left( \frac2{3 \sqrt{3}} - \frac1{36} \sqrt{\frac{8^3}{27}}  \right) \left( \frac12 \left( \frac{9}{8d} \right)^{1/6} |\tilde{A}|^{4/3} + \frac12  \left( \frac{9}{8d}\right)^{3/2} |B |^4  \right) \\
        &\geq c \left( |\tilde{A}|^{4/3} + |B |^4  \right),
    \end{align*}
    where $c> 0$ is a universal constant.
    \item If $\frac{8d}{9} |\tilde{A}| < |B |^3$ we find for $F = 3 B $ using $\frac{1}d |B |^4 \leq |B ^T B |^2 \leq |B |^4$ that
    \begin{align*}
        h(\tilde{A}, B ) &\geq - 3 |\tilde{A}| |B | + 27 |B ^T B |^2 - \frac{81}{4} |B ^T B |^2 \\
        &\geq \frac{27}{4d} |B |^4 - 3 |\tilde{A}| |B | \\
        &\geq \frac{27}{8d} |B |^4 \geq \frac{27}{16d} |B |^4 + \frac{27}{16d} \left( \frac{8d}{9} \right)^{4/3} |\tilde{A}|^{4/3} \geq c \left( |\tilde{A}|^{4/3} + |B |^4 \right),
    \end{align*}
    where $c>0$ is a universal constant.
\end{enumerate}

Consequently, we find for a constant $c>0$ (that depends on $\mathsf{G}$ and $\mathsf{L}$ and might change from line to line) that 
\begin{align}
g(A, a, B) &\geq c \left( |a|^{4/3} + |\tilde{A}|^{4/3} + |B |^4 \right) \nonumber \\
&\geq c \left( |a|^{4/3} + |A|^{4/3} - |B |^{4/3} + |B |^4 \right) \nonumber \\
&\geq c \left( |a|^{4/3} + |A|^{4/3} + |B |^4 \right) - \frac1c. \label{eq: est g}
\end{align}
\end{proof}

\begin{remark}
Similarly, one can show that it also holds
\[
g(A,a,B) \leq C \left( |a|^{4/3} + |A|^{4/3} + |B |^4 + 1 \right),
\]
which shows that the dual functional $I=\tilde{S}_H$ is well-defined and its minimization is non-trivial.
\end{remark}

Let us now show the coercivity of the dual functional $I=\tilde{S}_H$.

\begin{proposition} \label{prop: coercive st venant}
The functional $I$ as defined in \eqref{eq: def I} is weakly coercive i.e., for every sequence $(\mu_k,\lambda_k)_k \subseteq \{ (\mu,\lambda) \in L^{4/3}(\Omega;\R^{d \times d}) \times W^{1,4}(\Omega;\R^d): \operatorname{div}(\mu) \in L^{4/3}(\Omega;\R^d), \mu n = 0 \text{ on } \partial \Omega_p  \text{ and } \lambda= 0 \text{ on } \partial \Omega_y \}$ such that $\sup_k I(\mu_k,\lambda_k) < \infty$ there exists a (not relabeled) subsequence and $(\mu, \lambda) \in \{ (\mu,\lambda) \in L^{4/3}(\Omega;\R^{d \times d}) \times W^{1,4}(\Omega;\R^d): \operatorname{div}(\mu) \in L^{4/3}(\Omega;\R^d), \mu n = 0 \text{ on } \partial \Omega_p  \text{ and } \lambda= 0 \text{ on } \partial \Omega_y \}$ such that
\[
\mu_k \rightharpoonup \mu \text { in } L^{4/3}(\Omega; \R^{d\times d}), \; \operatorname{div}(\mu_k) \rightharpoonup \operatorname{div}(\mu) \text { in } L^{4/3}(\Omega; \R^{d}) \text{ and } \lambda_k \rightharpoonup \lambda \text{ in } W^{1,4}(\Omega;\R^d).
\]
\end{proposition}
\begin{proof}
Let us first recall Poincar\'e's inequality $\| \lambda\|_{W^{1,4}} \leq C \| \nabla \lambda\|_{L^4}$ for all $\lambda\in W^{1,4}(\Omega; \R^d)$ such that $\lambda= 0$ on $\partial \Omega_y$.
    Hence, we obtain from Proposition \ref{prop: lb g} that
\begin{align*}
    I(\mu_k,\lambda_k) \geq &c \left( \| \operatorname{div}(\mu_k) \|_{L^{4/3}}^{4/3} + \| \mu_k \|_{L^{4/3}}^{4/3} + \| \nabla \lambda_k\|_{L^4}^{4} \right) - \frac1c \mathcal{L}^d(\Omega)  \\ &-  \|\mu_k n \|_{(W^{1-1/4,4}(\partial \Omega))'} \| y^{(b)} \|_{W^{1-1/4,4}(\partial \Omega)} - \| \lambda_k\|_{L^4(\partial \Omega)} \| p^{(b)} \|_{L^{4/3}(\partial \Omega)}  \\
    \geq &c \left( \| \operatorname{div}(\mu_k) \|_{L^{4/3}}^{4/3} + \| \mu_k \|_{L^{4/3}}^{4/3} + \|  \lambda_k\|_{W^{1,4}}^{4} \right) - \frac1c \mathcal{L}^d(\Omega) \\ &- C\left( \| \operatorname{div}(\mu_k) \|_{L^{4/3}} + \| \mu_k \|_{L^{4/3}} + \| \lambda_k\|_{W^{1,4}} \right).
\end{align*}
Since $\sup_k I(\mu_k,\lambda_k) < \infty$, it follows immediately that $\sup_k \| \operatorname{div}(\mu_k) \|_{L^{4/3}} + \| \mu_k \|_{L^{4/3}} + \| \lambda_k \|_{W^{1,4}} < \infty$. This implies the existence of $\mu \in L^{4/3}(\Omega; \R^{d \times d})$, $\nu \in L^{4/3}(\Omega;\R^d)$ and $\lambda \in W^{1,4}(\Omega;\R^d)$ such that (up to a not relabeled subsequence) it holds $\mu_k \rightharpoonup \mu$ in $L^{4/3}$, $\operatorname{div}(\mu_k) \rightharpoonup \nu$ in $L^{4/3}(\Omega;\R^d)$ and $\lambda_k \rightharpoonup \lambda$ in $W^{1,4}(\Omega;\R^d)$. Then it is straightforward to check that $\operatorname{div}(\mu) = \nu \in L^{4/3}(\Omega; \R^d)$. Moreover, as the trace operators are weakly continuous it follows that $\mu  n = 0$ on $\partial \Omega_p$ and  $\lambda = 0$ on $\partial \Omega_y$ in the sense of traces as discussed at the beginning of this section.
\end{proof}

\subsection{Incompressible Neo-Hookean model in two dimensions}

Let us now consider the stress in an incompressible Neo-Hookean model
\[
P(F,p) = \operatorname{det}(F) \left(\mathsf{G} \, F - p F^{-T} \right),
\]
where $\mathsf{G} > 0$ is the shear modulus, $\operatorname{det}(F) = 1$ and $p$ is the pressure which is to be determined.
For simplicity, let $\mathsf{G} = 1$. Moreover, set
\[
\tilde{P}(F,p) = F - p \operatorname{cof}(F).
\]
Note that in two dimensions $\tilde{P}$ depends linearly on $F$ and for $F$ with $\operatorname{det}(F) = 1$ it holds $\tilde{P}(F,p) = P(F,p)$.
Then for $y^{(b)} \in W^{1,4}(\Omega;\R^d)$ and $p^{(b)} \in L^{4/3}(\partial \Omega;\R^d)$ we look for weak solutions to 
\[
\begin{cases}
F = \nabla y \text{ in } \Omega, \\
\operatorname{det}(F) = 1 \text{ in } \Omega, \\
\operatorname{div } \tilde{P}(F,p) = 0 \text{ in } \Omega, \\
\tilde{P}(F,p) \cdot n = p^{(b)} \text{ on } \partial \Omega_p, \\
y = y^{(b)} \text{ on } \partial \Omega_y.
\end{cases}
\]
Namely, we seek to find $y \in L^2(\Omega;\R^2)$ and $F \in L^2(\Omega; \R^{2\times 2})$ such that it holds for all $\mu \in L^2(\Omega; \R^{2\times 2})$ with $\operatorname{div}(\mu) \in L^2(\Omega; \R^2)$, $\lambda \in W^{1,4}(\Omega; \R^2)$, and $\vartheta \in L^{\infty}(\Omega)$ that
\[
\int_{\Omega} \operatorname{div}(\mu) \cdot y + \mu : F + \vartheta ( \operatorname{det}(F) - 1 ) + \nabla \lambda : \tilde{P}(F) \, dx  = \int_{\partial \Omega_p} \lambda \cdot p^{(b)} \,  d \mathcal{H}^1 +  \int_{\partial \Omega_y} y^{(b)} \cdot \mu n \,  d \mathcal{H}^1.
\]
As explained for the Saint Venant-Kirchhoff model, the boundary integrals have to be understood in the sense of traces.

Note that compared to the general discussion in Section \ref{sec:elastostat} the additional constraint $\operatorname{det}(F) = 1$ is treated through an extra dual variable $\vartheta$, whereas the pressure $p$ is an additional primal variable. 
Following the same approach, we define the function $g: \R^{d\times d} \times \R^d \times \R^{d\times d} \times \R \to \R  \cup \{+\infty\}$
\[
g(A,a,B,s) = \sup_{F \in \R^{d \times d}, y \in \R^d, p \in \R} a \cdot y + s ( \operatorname{det}(F) - 1 ) + A: F + B: P(F,p) - H(y,p,F),
\]
where $H(y,p,F) = \frac12 y^2 + \frac12 |F|^2 + \frac14 p^4$.
Correspondingly, we define the functional 
\begin{equation*}
\begin{aligned}
    I: \big\{ (\mu,\lambda,\vartheta) \in L^{4/3}(\Omega;\R^{d \times d}) & \times W^{1,4}(\Omega;\R^d) \times L^{\infty}(\Omega):\\
    & \operatorname{div}(\mu) \in L^{4/3}(\Omega;\R^d), \mu n = 0 \text{ on } \partial \Omega_p  \text{ and } \lambda= 0 \text{ on } \partial \Omega_y \big\} \to \R 
\end{aligned}
\end{equation*}
by
\begin{equation} \label{eq: def I neo hook}
I(\mu,\lambda,\vartheta) = \int_{\Omega} g(\mu,\operatorname{div}(\mu),\nabla \lambda,\vartheta) \, dx - \int_{\partial \Omega_y}  y^{(b)}\cdot \mu n \, d\mathcal{H}^{d-1} - \int_{\partial \Omega_p} \lambda\cdot p^{(b)} \, \mathcal{H}^{d-1}.
\end{equation}
By the same argument as for the Saint Venant - Kirchhoff material in Proposition \ref{prop: SH = I} it holds for all $(\mu,\lambda,\vartheta) \in \{ (\mu,\lambda,\vartheta) \in L^{4/3}(\Omega;\R^{d \times d}) \times W^{1,4}(\Omega;\R^d) \times L^{\infty}(\Omega): \operatorname{div}(\mu) \in L^{4/3}(\Omega;\R^d), \mu n = 0 \text{ on } \partial \Omega_p  \text{ and } \lambda= 0 \text{ on } \partial \Omega_y \}$ that
\begin{align}
&I(\mu,\lambda,\vartheta) \\ = &\sup_{y \in L^2(\Omega;\R^2), F \in L^2(\Omega;\R^{2\times 2}), p \in L^4(\Omega)} \int_{\Omega} \operatorname{div}(\mu) \cdot y + \vartheta ( \operatorname{det}(F) - 1 ) + \mu: F + \nabla \lambda: P(F,p) - H(y,p,F) \, dx \nonumber \\& - \int_{\partial \Omega_p} \lambda \cdot p^{(b)} \,  d \mathcal{H}^1 -  \int_{\partial \Omega_y} y^{(b)} \cdot \mu n \,  d \mathcal{H}^1. \label{eq: def I neo hook2}
\end{align}
The right hand side of the above equation is the analogue of $\tilde{S}_H$ including the extra variable $\vartheta$.

As in the section above, in order to show existence of minimizers of $I$ it suffices to show that $I$ is coercive. 
Again, we first prove lower bounds for the function $g$.

\begin{proposition}\label{prop: lb neo hook}
There exist constants $C > c >0$ such that the following is true.
\begin{enumerate}
    \item If $|s| > 1$ then $g(A,a,B,s) = +\infty$ for all $(A,a,B) \in \R^{d \times d} \times \R^d \times \R^{d\times d}$. \label{item1}
    \item If $|s| \leq 1$ then it holds for $(A,a,B) \in \R^{d \times d} \times \R^d \times \R^{d\times d}$ \label{item2}
    \[
    g(A,a,B,s) \geq c (|a|^2 + |A + B|^2 + |B|^4 ) - 1
    \]
    \item If $|s| \leq 1$ then it holds for $(A,a,B) \in \R^{d \times d} \times \R^d \times \R^{d\times d}$ that \label{item3}
    \[
    g(A,a,B,s) \leq C \left( |a|^2 +  \frac1{1-|s|}|A+B|^2 + \frac{1}{(1-|s|)^2} |B|^4 + 1 \right).
    \]
\end{enumerate}
\end{proposition}
\begin{proof}
First note that by optimizing in $y$ we find
\begin{align}
    g(A,a,B,s) = &\sup_{F \in \R^{d \times d}, y \in \R^d, p \in \R} a \cdot y + s ( \operatorname{det}(F) - 1 ) + A: F + B: P(F,p) - H(y,p,F) \nonumber \\
    =& \frac12 |a|^2 - s + \sup_{F \in \R^{d \times d}, p \in \R} s  \operatorname{det}(F)  + A: F + B: P(F,p) - \frac14 p^4 - \frac12 |F|^2 \nonumber \\
    =:& \frac12 |a|^2 - s + h(A,B,s). \label{eq: formula g}
\end{align}
We will now only work with the function $h$. 
We observe that it holds for all $F\in \R^{2\times 2}$ that $F : \operatorname{cof}(F) = 2 \operatorname{det}(F)$. 
In particular, $|\operatorname{\det}(F)| \leq \frac12 |F|^2$.
Let us now assume that $s>1$. Then consider for $t > 0$ the matrix $F = t Id$. In addition, set $p=0$. It follows that 
\[
h(A,B,s) \geq s t^2 - 2 |A+B| t - t^2 \stackrel{t \to \infty}{\longrightarrow} + \infty.
\]
The case $s<-1$ can be treated similarly.  This shows \ref{item1}.
Next, let us consider $|s|\leq 1$. First we assume that $\frac1{4\sqrt{8}} |B|^2 \leq |A+B|$. Then set $F = (A + B)/2$ and $p=0$ to obtain
\begin{align} \label{eq: A+B large}
    h(A,B,s) \geq (A+B): F - \frac12 |F|^2 + s \operatorname{det}(F) \geq \frac{|A+B|^2}2 - |F|^2 &= \frac{|A+B|^2}4  \\ &\geq \frac1{512} \left( |A+B|^2 + |B|^4 \right). \nonumber
\end{align}
For $|A+B| \leq \frac1{4\sqrt{8}}|B|^2$ consider $F \in \R^{2 \times 2}$ such that $\operatorname{cof}(F) =- \frac1{\sqrt8} |B| B$ and $p =  \frac1{\sqrt{2}} |B|$. Then
\begin{align}
    h(A,B,s) &\geq (A+B):F - B:(p \operatorname{cof}(F)) + s \operatorname{det}(F) - \frac14 p^4 - \frac12 |F|^2 \nonumber \\
    &\geq (A+B):F - B:(p \operatorname{cof}(F)) - \frac14 p^4 - |F|^2 \nonumber \\
    &\geq - \frac1{\sqrt{8}} |A+B| \, |B|^2 + \frac1{4} |B|^4 - \frac1{16} |B|^4 - \frac1{8} |B|^4 \nonumber \\
    &\geq - \frac1{32} |B|^4 + \frac1{16} |B|^4 = \frac1{32} |B|^4 \geq \frac1{64} \left( |A+B|^2 + |B|^4 \right) \label{eq: B large}
\end{align}
Combining \eqref{eq: formula g}, \eqref{eq: A+B large} and \eqref{eq: B large} yields 
\[
g(A,a,B,s) \geq \frac12 |a|^2 - |s| + h(A,B,s) \geq c ( |a|^2 + |A+B|^2 + |B|^4) - 1,
\]
which is \ref{item2}.

It remains to show \ref{item3}. 
We estimate
\begin{align*}
h(A,B,s) &\leq \sup_{F \in \R^{2 \times 2}, p \in \R} F : ( A + B - p \operatorname{cof}(B)) - \frac14 p^4 - \frac{1 - |s|}{2} |F|^2 \\
&= \sup_{p \in \R} \frac1{2(1-|s|)} |A + B - p \operatorname{cof}(B)|^2 - \frac14 p^4 \\
&\leq \frac1{(1-|s|)} |A+B|^2 + \sup_{p \in \R} \frac{p^2}{(1-|s|)} |\operatorname{cof}(B)|^2 - \frac14 p^4 \\
&\leq \frac1{(1-|s|)} |A+B|^2 + \frac{1}{(1-|s|)^2} |B|^4.
\end{align*}
Combining this with \eqref{eq: formula g} yields \ref{item3}.
\end{proof}

Eventually, we state the coercivity statement for the dual functional $I$.

\begin{proposition}
The functional $I$ as defined in \eqref{eq: def I neo hook} is weakly coercive i.e., for every sequence $(\mu_k,\lambda_k,\vartheta_k)_k \subseteq \{ (\mu,\lambda,\vartheta) \in L^{4/3}(\Omega;\R^{d \times d}) \times W^{1,4}(\Omega;\R^d) \times L^{\infty}(\Omega): \operatorname{div}(\mu) \in L^{4/3}(\Omega;\R^d), \mu n = 0 \text{ on } \partial \Omega_p  \text{ and } \lambda= 0 \text{ on } \partial \Omega_y \}$ such that $\sup_k I(\mu_k,\lambda_k,\vartheta_k) < \infty$ there exists a (not relabeled) subsequence and $(\mu, \lambda,\vartheta) \in \{ (\mu,\lambda,\vartheta) \in L^{4/3}(\Omega;\R^{d \times d}) \times W^{1,4}(\Omega;\R^d) \times L^{\infty}(\Omega): \operatorname{div}(\mu) \in L^{4/3}(\Omega;\R^d), \mu n = 0 \text{ on } \partial \Omega_p  \text{ and } \lambda= 0 \text{ on } \partial \Omega_y \}$ such that
\begin{align*}
&\mu_k \rightharpoonup \mu \text { in } L^{2}(\Omega; \R^{d\times d}), \; \operatorname{div}(\mu_k) \rightharpoonup \operatorname{div}(\mu) \text { in } L^{2}(\Omega; \R^{d}), \; \lambda_k \rightharpoonup \lambda \text{ in } W^{1,4}(\Omega;\R^d) \\ &\text{ and } \vartheta_k \stackrel{*}{\rightharpoonup} \vartheta \text{ in } L^{\infty}(\Omega).
\end{align*}    
\end{proposition}
\begin{proof}
    The proof is similar to the proof of Proposition \ref{prop: coercive st venant} using Proposition \ref{prop: lb neo hook} \ref{item1}.~and \ref{item2}.
\end{proof}

We end this section with a brief comment on changing the base states.
\begin{remark}
Let us fix $\tilde{y} \in L^2(\Omega;\R^2)$ and $\tilde{F} \in L^2(\Omega;\R^{2\times 2})$ and note for $\tilde{H}(F,y,p;x) = H(F - \tilde{F}(x), y - \tilde{y}(x),p)$ that
\begin{align*}
&S_{\tilde{H}}(\lambda,\mu,\vartheta) \\:= &\sup_{y \in L^2, F \in L^2, p \in L^4} \int_{\Omega} \operatorname{div}(\mu) \cdot y + \vartheta ( \operatorname{det}(F) - 1 ) + \mu: F + \nabla \lambda: P(F,p) - \tilde{H}(F,y,p;x) \, dx \\& - \int_{\partial \Omega_p} \lambda \cdot p^{(b)} \,  d \mathcal{H}^1 -  \int_{\partial \Omega_y} y^{(b)} \cdot \mu n \,  d \mathcal{H}^1 \\
=& \sup_{y \in L^2, F \in L^2, p \in L^4} \int_{\Omega} \operatorname{div}(\mu) \cdot y + \vartheta ( \operatorname{det}(F) - 1 ) + \mu: F + \nabla \lambda: P(F,p) - \tilde{H}(F,y,p;x) \, dx \\& - \int_{\partial \Omega_p} \lambda \cdot p^{(b)} \,  d \mathcal{H}^1 -  \int_{\partial \Omega_y} y^{(b)} \cdot \mu n \,  d \mathcal{H}^1 \\
=& \sup_{y \in L^2, F \in L^2, p \in L^4} \int_{\Omega} (\operatorname{div}(\mu)-\tilde{y}) \cdot y + \vartheta ( \operatorname{det}(F) - 1 ) + (\mu + \nabla \lambda - p\operatorname{cof}(\nabla \lambda)- \tilde{F}): F - H(F,y,p) \, dx \\& + \int_{\Omega} \frac12 |\tilde{F}|^2 + \frac12 |\tilde{y}|^2 \, dx - \int_{\partial \Omega_p} \lambda \cdot p^{(b)} \,  d \mathcal{H}^1 -  \int_{\partial \Omega_y} y^{(b)} \cdot \mu n \,  d \mathcal{H}^1 \\
=& \int_{\Omega} g(\mu - \tilde{F},\operatorname{div}(\mu)-\tilde{y},\nabla \lambda,\vartheta) \, dx - \int_{\partial \Omega_y}  y^{(b)}\cdot \mu n \, d\mathcal{H}^{d-1} - \int_{\partial \Omega_p} \lambda\cdot p^{(b)} \, \mathcal{H}^{d-1} + \int_{\Omega} \frac12 |\tilde{F}|^2 + \frac12 |\tilde{y}|^2 \, dx.
\end{align*}
This functional has the same coercivity and lower-semicontinuity properties as $\tilde{S}_H$. Hence, existence of minimizers for this functional is also guaranteed. The question of whether such a minimzer gives rise to a weak solution of the primal problem hinges on the regularity of the function $g$ on the values of a minimizer. However, Proposition \ref{prop: lb neo hook} shows that at least at $\vartheta = \pm 1$ the function $g$ is not differentiable. Making good choices for $\tilde{F}$ and $\tilde{y}$ could enable us to guarantee that a minimizer $(\lambda,\mu,\vartheta)$ of $\tilde{S}_{\tilde{H}}$ only takes values such that $g$ is actually differentiable in the points $(\mu - \tilde{F}, \operatorname{div}(\mu) - \tilde{y},\nabla \lambda,\vartheta)$. 
\end{remark}
\section{A computational case study: non-convex elasticity in 1-d}\label{sec:compute}
In this section, we present computational approximations of our dual scheme applied to nonconvex elastostatics and elastodynamics in 1 space dimension and time; this provides intuition for its capabilities and evidence for its practical feasibility.

 The specific variational and PDE problems we study in the elastostatic case are described in Sec.~\ref{sec:comp_form}.  The finite element method based algorithm used to solve the problem is developed in Sec.~\ref{subsec:alg_stat}. Results of the computational studies are presented and discussed in Sec.~\ref{sec:results}. In particular, in Sec.~\ref{sub_sec:case_5} is presented and discussed the stability of some dual solutions to elastodynamic problems that are `Hadamard unstable' in the primal, classical elasticity formulation, in the sense of lacking continuous dependence on initial data because some regions of the domain have initial conditions where hyperbolicity is lost. While it is \textit{not} our goal to discuss the (de)merits of conventional elasticity theory to model such situations, we give a physical interpretation of the problem that is studied. This section also contains a brief description of the extension of the computational scheme to elastodynamics.

\subsection{Formulation}\label{sec:comp_form}

We computationally approximate solutions to a double-well elastostatics problem in 1 space dimension without any higher-order (strain-gradient) regularization. The objective  is to compute critical points of the following energy functional:
\begin{equation}\label{eqn:primal_func}
 I = \int_0^1 \left( \left( u_x -1\right)^2 - 1 \right)^2 + \left (u - \alpha x\right)^2 \,dx; \qquad u(0) = 0; \quad u(1) = \alpha^*.
\end{equation}
We use the notation $\p_i a = a_i, i = x,t$, and $\alpha, \alpha^*$ are specified constants. The function $u$ corresponds to the displacement. The Euler-Lagrange equations of \eqref{eqn:primal_func} are
\begin{subequations}
    \label{eqn:primal_EL}
    \begin{align}
        u_x - e & = 0 \\
        2\p_x \left((e-1)^3 - (e-1)\right) - \left(u - \alpha x\right) & = 0 \\
        u(0) = 0  &; \qquad  u(1) = \alpha^*,
     \end{align}
\end{subequations}
where $e$ corresponds to the strain. The stress, $\sigma$, for this model is given by
\[
\sigma = 4(e-1)\left((e - 1)^2 -1)\right).
\]

Using the dual (Lagrange multiplier) fields $\lambda$ and $\mu$ for the two equations, we define the pre-dual functional 
\begin{equation*}\label{eqn:pre_dual_ex}
     \widehat{S}\, \left[u,\lambda,e,\mu \right] \,=\, \int_0^1\,\, (-u\lambda_x - \lambda e - 2((e-1)^3 - (e-1)) \mu_x - \mu(u - \alpha x) + H(u,e))\,dx + \alpha^* \lambda(1),
\end{equation*}
where the requirement on the function $H$ is that, defining 
\[
\dee := \left( \lambda, \lambda_x, \mu, \mu_x \right),
\]
the following equations are solvable for $u$, $e$ in terms of $\dee:$
\begin{equation}
    \label{eq:mapping}
    \begin{aligned}
        -\lambda_x - \mu + \p_u H &=0 \\
        -2(3(e-1)^2-1)\mu_x - \lambda + \p_e H &=0.
    \end{aligned}
\end{equation}
Choosing a ‘shifted quadratic’ and ‘shifted cubic’ form for the potential $H$:
\begin{equation*}
H(u, e; x) = \frac{1}{2}\left( c_u(u-\bar{u}(x))^2\right) +\frac{1}{2}\left(c_e(e - \bar{e}(x))^2\right)
+\frac{1}{3}\left(c_e(e - \bar{e}(x))^3\right),
\end{equation*}
we obtain the following \textit{dual-to-primal} (DtP) mapping equations
\begin{equation}
    \label{eq:Dtp}
    \begin{aligned}
        & u^{(H)}(\lambda_x,\mu,x) = \frac{ \lambda_x + \mu}{c_u} + \bar{u}(x) \\
        & -6\left(e^{(H)}(\mu_x, \lambda,x) -1 \right)^2\mu_x + 2\mu_x -\lambda \\
        & + c_e \left(e^{(H)}(\mu_x, \lambda,x)-\bar{e}(x) \right) \left|e^{(H)}(\mu_x, \lambda,x) - \bar{e}(x)\right| + c_e\left(e^{(H)}(\mu_x, \lambda,x) - \bar{e}(x)\right) = 0.
    \end{aligned}
\end{equation}
\textit{Henceforth, the function $u^{(H)}$ will be referred to as $\hat{u}$ with the understanding that the symbol carries argument $\dee$ and $\bar{u}$. This also applies to $e^{(H)}$}:
\begin{equation*}
    \hat{u}(x) := u^{(H)}(\lambda_x(x), \mu(x), x); \qquad \hat{e}(x): = e^{(H)}(\mu_x(x), \lambda(x), x).
\end{equation*}
It is clear from \eqref{eq:mapping} and \eqref{eq:Dtp} that $\hat{u} = \bar{u}$ and $\hat{e} = \bar{e}$ is a value of the DtP mapping for $\dee =0$.

We substitute the DtP mapping \eqref{eq:Dtp} in the functional $\widehat{S}$ to define the dual functional 
\begin{equation}
\label{eqn:dual_exact}
     S\, \left[\lambda, \mu\right] \,=\, \int_0^1\,\, \left(-\hat{u}\lambda_x - \lambda \hat{e} - 2\left((\hat{e}-1)^3 - (\hat{e}-1)\right) \mu_x - \mu\left(\hat{u} - \alpha x\right) + H\left(\hat{u},\hat{e}\right)\right)\,dx + \alpha^*\lambda(1).
\end{equation}
The requirement \eqref{eq:mapping} ensures that the dual E-L system of \eqref{eqn:dual_exact} is simply the primal system \eqref{eqn:primal_EL}, with the replacements $u \to \hat{u}, e \to \hat{e}$. Thus, the critical points for the dual functional $S$ \eqref{eqn:dual_exact} is also a critical point of the primal functional $I$ \eqref{eqn:primal_func}, interpreted through the DtP mapping. The algorithm employed to obtain such critical points is discussed next.

\subsection{Finite element algorithm}\label{subsec:alg_stat}
Denoting the dual fields as
\[
D = (D_1, D_2) := (\lambda, \mu),
\]
the following weak form suffices to compute solutions to the dual problem \eqref{eqn:weak_form}, which also constitutes the first variation of the dual functional \eqref{eqn:dual_exact} set to 0:
\begin{equation}
\label{eqn:weak_form}
    \begin{aligned}
      R[D; \delta D]  & = -\int_0^1\,\hat{u}\delta \lambda_x  \,dx \,- \int_0^1\,\hat{e}\delta \lambda\,dx\\
      & \quad - \int_0^1\,2\left(\left(\hat{e}-1\right)^3 - (\hat{e}-1)\right) \delta\mu_x\,dx  \,-
     \int_0^1\,\left(\hat{u}- \alpha x\right) \delta\mu \,dx + \alpha^*\delta\lambda(1) = 0;\\
     & \qquad \mu(0) = 0, \qquad \mu(1) = 0.
     \end{aligned}
\end{equation}
A Dirichlet boundary condition is applied only on $\mu$. A Neumann boundary condition is applied on $\lambda$, corresponding to the Dirichlet boundary condition on $u$ in the primal problem. 

We use the Finite Element (FE) method to discretize the problem. A linear span of globally continuous, piecewise smooth finite element shape functions corresponding to a FE mesh for the domain  is used to achieve this discretization. These shape functions are represented by $N^{\left(\cdot\right)}$, where ($\cdot$) denotes the index of the node under consideration. We use $C^0$ FE shape functions to approximate the dual fields. Hence, the direct evaluation of primal fields through the DtP mapping exhibit discontinuities in the domain in general, even when approximating continuous primal solutions. This is because the DtP mappings involve derivatives of the dual fields. To deal with this feature, we use an $L^2$ projection  of the DtP mapping generated primal $u$ field from the dual solution to define the corresponding continuous primal field.

The discretized dual fields and test functions are expressed as
\[
D_i(x) = D^A_i N^A(x); \qquad   \delta D_i(x) = \delta D^A_i N^A(x); \qquad dD_i(x) = dD^A_i N^A(x), \qquad i = 1,2,
\]
where $D^A_i$ denote the finite element nodal degrees of freedom (which are the coefficients of a linear combination of  a finite dimensional representation of the dual fields). The discretized version of \eqref{eqn:weak_form} generates a discrete residual $R^A_i(\dee)$ given by 
\begin{equation*}
    R[D; \delta D] = \delta D^A_i R^A_i(\dee)
\end{equation*}
and its variation in the direction $dD$, the Jacobian
\begin{equation*}
    J[D; \delta D, d D] = \delta D^A_i \, J^{AB}_{ij}(\dee) \, dD^B_j
\end{equation*}
given by
\begin{equation}
\label{eqn:jacobian}
    \begin{aligned}
     J[D;\delta D,dD]\, =&\,  -\int_0^1\, \delta \lambda_x \left[\frac{\p \hat{u}}{\p \lambda_x}\,d\lambda_x + \frac{\p \hat{u}}{\p \mu}\,d\mu \right] \,dx \,- 
     \int_0^1\,\delta \lambda  \left[\frac{\p \hat{e}}{\p \mu_x}\,d\mu_x + \frac{\p \hat{e}}{\p \lambda}\,d\lambda \right]\,dx\\
     &\,- \int_0^1\,2\left(3\left(\hat{e}-1\right)^2 - 1\right)\delta\mu_x\left[\frac{\p \hat{e}}{\p \mu_x}\,d \mu_x + \frac{\p \hat{e}}{\p \lambda}\,d\lambda \right] \,dx \\\,-
     &\int_0^1\,\delta\mu \left[\frac{\p \hat{u}}{\p \lambda_x}\,d\lambda_x + \frac{\p \hat{u}}{\p \mu}\,d\mu \right] \,dx\\
     \end{aligned}
\end{equation}
that are used in a Newton Raphson scheme. Here, $R^A_i, J^{AB}_{ij}$ are given by the expressions
\begin{equation*}
    \begin{aligned} 
    &R^A_1 = \,  -\int_0^1\,\hat{u} N^A_x  \,dx\,- \int_0^1\,\hat{e}N^A \,dx\\
    &R^A_2 = - \int_0^1\,2\left(\left(\hat{e}-1\right)^3 - (\hat{e}-1)\right) N^A_x\,dx  \,-\int_0^1\,\left(\hat{u} - \alpha x\right) N^A \,dx   
     \end{aligned}
\end{equation*}
and
\begin{equation*}
    \begin{aligned}
    \\&J^{AB}_{11} = \,  -\int_0^1\,\frac{\p \hat{u}}{\p \lambda_x}N^A_xN^B_x  \,dx\,- \int_0^1\,\frac{\p \hat{e}}{\p \lambda}N^AN^B \,dx\\
    &J^{AB}_{12} = \,  -\int_0^1\,\frac{\p \hat{e}}{\p \mu_x}N^AN^B_x  \,dx\,- \int_0^1\,\frac{\p \hat{u}}{\p \mu}N^A_xN^B \,dx\\
    &J^{AB}_{21} = \,  -\int_0^1\,\frac{\p \hat{u}}{\p \lambda_x}N^AN^B_x  \,dx\,- \int_0^1\,2(3(\hat{e}-1)^2 - 1)\,\frac{\p \hat{e}}{\p \lambda}N^A_xN^B \,dx\\
    &J^{AB}_{22} = \,  -\int_0^1\,2(3(\hat{e}-1)^2 - 1)\,\frac{\p \hat{e}}{\p \mu_x}N^A_xN^B_x  \,dx\,- \int_0^1\,\frac{\p \hat{u}}{\p \mu}N^AN^B \,dx\\
     \end{aligned}
\end{equation*}

The Newton-Raphson scheme involves solving the following matrix equation for the $k^{th}$ correction to an initial guess $D^{B(0)}_j$
\begin{equation*}
    \begin{aligned}
     - R^A_i\left(\dee^{(k-1)}\right) & =  J^{AB}_{ij}\left(\dee^{(k-1)}\right)\, dD^B_j\\
         D^{B(k)}_j & = D^{B(k-1)}_j + dD^B_j,
    \end{aligned}
\end{equation*}
and this is continued until convergence of $|R(\dee)|$ to $0$, up to the convergence threshold $tol$.

An initial guess for the dual fields $D$ is required by the Newton Raphson (NR) algorithm. It is practical, and more feasible, to make an initial guess for a primal fields $(\hat{u},\hat{e})$ with physical meaning, and infer initial guesses for the dual fields consistent with it as a solution of the DtP mapping.  Choosing initial guesses $\hat{u}_{(i)}$ and $\hat{e}_{(i)}$
\begin{equation*}
\hat{u}_{(i)} = \bar{u}, \qquad \hat{e}_{(i)} = \bar{e},
\end{equation*}
the initial guess for the dual fields, 
\begin{equation*}
D_i^B = 0,
\end{equation*}
satisfies the DtP mapping equations and this is considered by us as a good initial guess to initiate the Newton-Raphson iterations. The algorithm capturing these ideas has been outlined in Table.~\ref{algo:euler_algorithm}.  

Some notational details related to our computational algorithm and results are as follows: 
\begin{itemize}
\item $tol$ represents the tolerance for the iterative solve of the Newton Raphson iterations; residuals
with $|R(\dee)| < tol$ are accepted as having approximately satisfied the corresponding
equations. We use the definition $|R(\dee)| := \max_{A,i} R^A_i(\dee)$, where $(A, i)$ ranges over all
dual nodal degrees of freedom on which a Dirichlet boundary conditions are not specified.
\item $u^{(t)}, e^{(t)}$ will represent exact `target' solutions for the context under discussion.
\item We report the accuracy of our approximations in the $L^1$ norm:
\begin{equation*}
    \|f\|_1 = \int_{0}^{1} |f(x)| \, dx.
\end{equation*}
\item In all our calculations we seek weak solutions of the E-L system so that jump conditions (like traction continuity) are automatically imposed.
\end{itemize}

\begin{table}
{\begin{algorithm}[H]
\caption*{\textbf{Algorithm}} 
\begin{algorithmic}
\State \textbf{Initialization}: Set $D^{B(0)}_j = 0$. Choose a value for $tol$. 
\\\hrulefill
\begin{enumerate}
    \item Newton's Method is used to evaluate the results. Set $k=1$
    \item[] \textbf{For} $k \geq 1$:
    \begin{enumerate}[i]
        \item \label{step-i} Evaluate $R^A_i\left(\dee^{(k-1)}\right)$ and $J^{AB}_{ij}\left(\dee^{(k-1)}\right)$.
        \item Solve $- R^A_i\left(\dee^{(k-1)}\right)  =  J^{AB}_{ij}\left(\dee^{(k-1)}\right)\, dD^B_j$, with appropriate boundary conditions. 
        \item Set $D^{B(k)}_j = D^{B(k-1)}_j + dD^B_j$.
        \item Evaluate $d^{(k)} = \max_{(A,i)}\,|R^A_i\left(\dee^{(k)}\right)|$. 
        \item[]\textbf{if} $d^{(k)}<tol$ \textbf{then} go to step \ref{stepout}
        \item[]\textbf{else do} $k=k+1$ and go to step \ref{step-i}   
    \end{enumerate}
    \item \label{stepout} Evaluate the values of $\hat{u}$ and $\hat{e}$ at the nodes using the $L^2$ projection. These nodal values establish the solution.
\end{enumerate}
\end{algorithmic}
\end{algorithm}}
\caption{Algorithm to solve non-convex elastostatics in 1-d.}
\label{algo:euler_algorithm}
\end{table}
\subsection{Results}\label{sec:results}
The algorithm of Sec.~\ref{subsec:alg_stat} is utilized to solve key problems demonstrating the theory. It is well-understood that due to the lack of lower semi-continuity of the primal energy functional \eqref{eqn:primal_func}, limits of energy minimizing sequences do not, in general, minimize the energy, i.e.~the minimization problem does not have any solution (for an elementary treatment, see \cite{bhattacharya2003microstructure}). The dual scheme, on the other hand, defines the notion of a variational dual solution to the E-L equations of the energy functional, as in Definition \ref{algo:euler_algorithm}, Sec.~\ref{sec:var_anal} . Such solutions need not have a connection to being minimizers of the primal energy functional, and we compute such results in the following examples.
Before proceeding to the problems that follow, we make the observation that our goal is to seek solutions to the system \eqref{eqn:primal_EL}. In particular, homogeneous strain profiles $u_x = \beta$  are solutions provided $\beta = \alpha = \alpha^*$.
\subsubsection{Stress-free solution}
We consider the primal energy functional 
\begin{equation*}\label{eqn:case_1_func}
 I = \int_0^1 \left( \left( u_x - 1\right)^2 - 1 \right)^2 + (u - x)^2 \,dx; \qquad u(0) = 0; \quad u(1) = 1\
\end{equation*}
and compute solutions to its Euler-Lagrange equation \eqref{eqn:primal_EL} with $\alpha = \alpha^* = 1$.
A stress-free equilibrium solution of the primal problem is given by
\[
u^{(t)} = x, \qquad e^{(t)} = 1,
\]
considered a target solution here, and we seek to compute it by the dual scheme, probing its `basin of attraction' by choosing initial guesses sufficiently away from it.
The boundary conditions chosen  (profile shown in Fig.~\ref{fig:case_1_s}) accommodate the stress-free configuration (which happens to be a limit of an energy minimizing sequence). As mentioned in Sec.~\ref{subsec:alg_stat}, we choose the initial guess to coincide with the base state; we demonstrate the solution of two dual problems, each with a different choice of base state as shown in Fig.~\ref{fig:case_1_s_sf_a} and Fig.~\ref{fig:case_1_s_sf_b}. The base state provided for an $\bar{e}$ profile is the derivative of the corresponding $\bar{u}$ profile. The $\bar{u}$ profiles satisfy the displacement b.c.s of the target solution. The dual scheme is able to recover the target solution shown in Fig.~\ref{fig:case_1_s} without difficulty, even for an initial guess amplitude that is $30\%$ away from the target. The $L^1$ displacement and strain errors relative to the target solution is summarized in Table.~\ref{tab:case_1} for the case corresponding to Fig.~\ref{fig:case_1_s_sf_a}. These values are of the same order for the other sub-case (Fig.~\ref{fig:case_1_s_sf_b}) as well.
\begin{figure}[H]
    \centering
    \begin{subfigure}[b]{0.4\textwidth}
        \centering
        \includegraphics[scale=0.40]{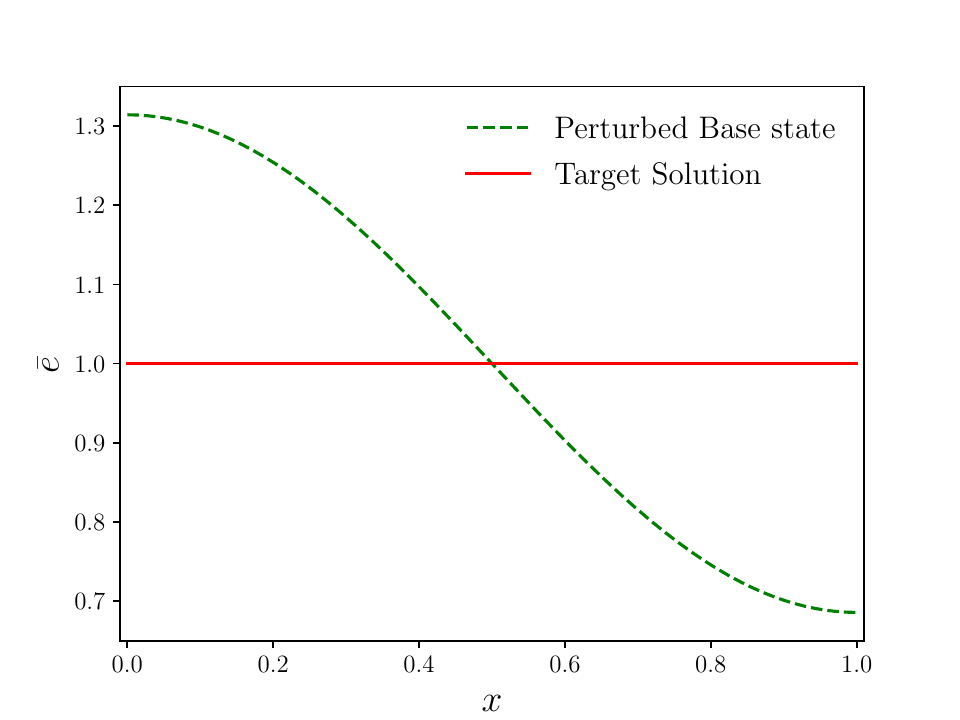}
        \caption{}
        \label{fig:case_1_s_sf_a}
    \end{subfigure}
    \begin{subfigure}[b]{0.4\textwidth}
        \centering
        \includegraphics[scale=0.40]{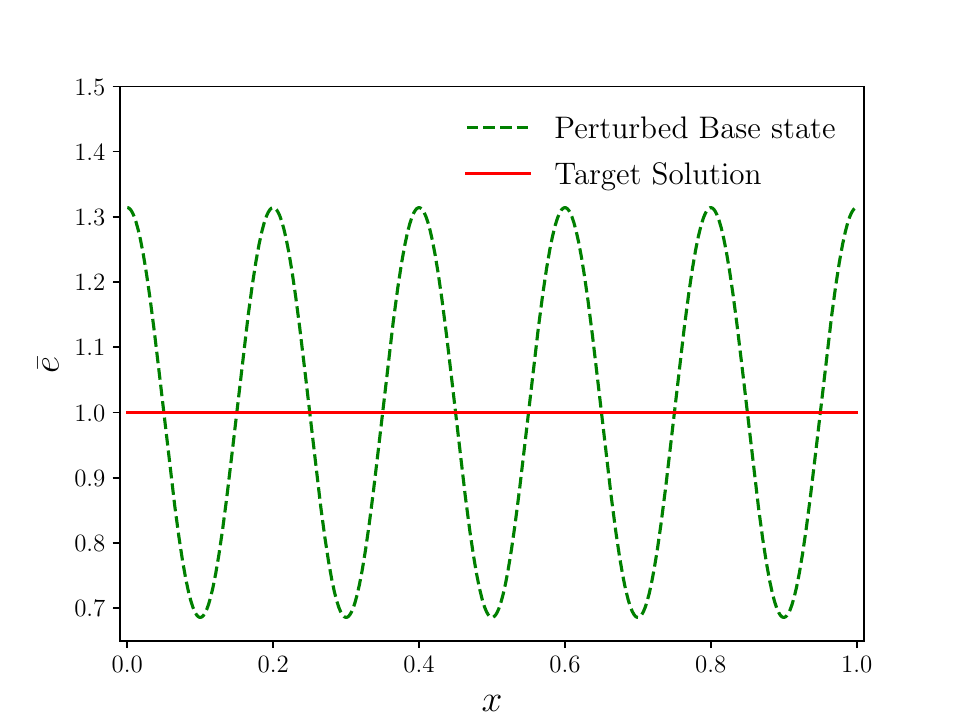}
        \caption{}
        \label{fig:case_1_s_sf_b}
    \end{subfigure}
    \begin{subfigure}[b]{0.5\textwidth}
        \centering
        \includegraphics[scale=0.50]{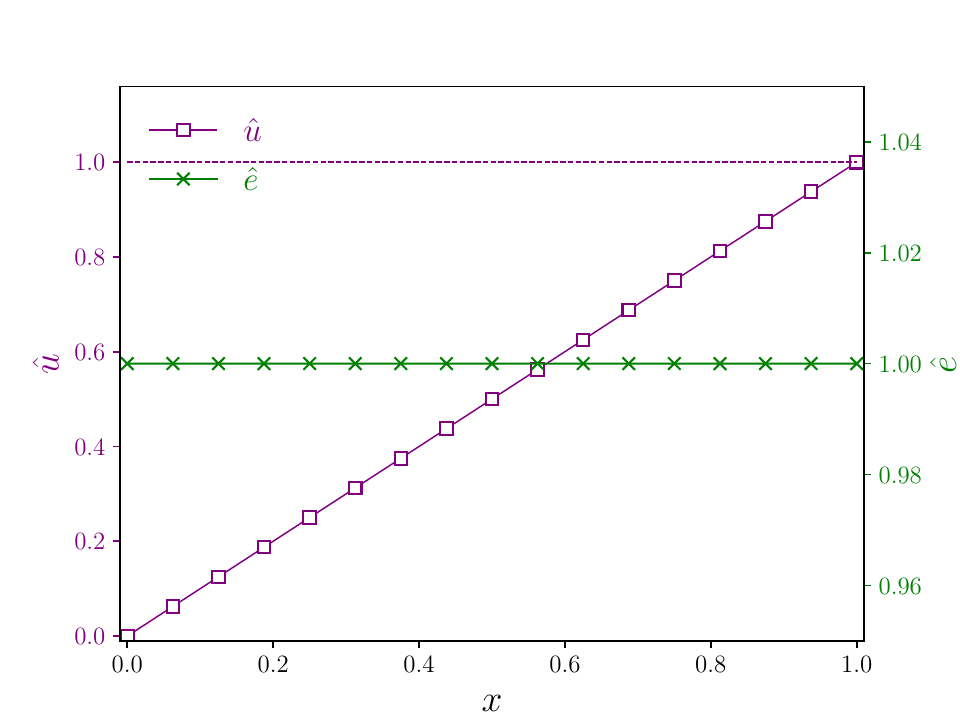}
        \caption{}
        \label{fig:case_1_s_sf_c}
    \end{subfigure}
    \caption{ (a) and (b) show the initial guesses, $\bar{e}$, for two cases and (c) shows the solution for the primal fields. Fig.~(c) contains both $\hat{u}$ and $\hat{e}$ fields with different y-axis labels. The color purple corresponds to information related to the displacement, while the color green corresponds to to that for strain.}
    \label{fig:case_1_s}
\end{figure}
\begin{table}[h]
\centering
\begin{tabular}{|c|c|c|} 
\hline
\textbf{Number of Elements} & \textbf{$\|\hat{u}-u^{(t)}\|_1$} & {$\|\hat{e}-e^{(t)}\|_1$} \\
\hline
100 & $1\times10^{-4}$ & $1\times10^{-5}$\\
1600 & $4\times10^{-7}$ & $2\times10^{-7}$\\
8000 & $1\times10^{-8}$ & $8\times10^{-9}$\\
\hline
\end{tabular}
\caption{Error reduction in $L^1$ norms of displacement and strain with mesh refinement. These values are for the case corresponding to Fig.~\ref{fig:case_1_s_sf_a}.}
\label{tab:case_1}
\end{table}
\subsubsection{Stressed solutions with spatially homogeneous and inhomogeneous strain configurations}
We now consider the primal functional 
\begin{equation}\label{eqn:case_2_func}
 I = \int_0^1 \left( \left( u_x - 1\right)^2 - 1 \right)^2 + (u - 0.5x)^2 \,dx; \qquad u(0) = 0; \quad u(1) = 0.5,
\end{equation}
with
$$
e^{(t)} = 0.5, \qquad u^{(t)} = 0.5x,
$$ 
as the exact target solution. We choose a base state near this solution as shown in Fig.~\ref{fig:case_2_sf_a} and start from an initial guess coincident with the base state. The base state provided for an $\bar{e}$ profile is the derivative of the corresponding $\bar{u}$ profile. The $\bar{u}$ profiles satisfy the displacement b.c.s of the target solution. The dual scheme is able to recover the target solution, presented in Fig.~\ref{fig:case_2_sf_b}. The $L^1$ displacement and strain errors relative to the target solution are summarized in Table.~\ref{tab:case_2}.
\begin{figure}[H]
    \centering
    \begin{subfigure}[b]{0.45\textwidth}
        \centering
        \includegraphics[scale=.45]{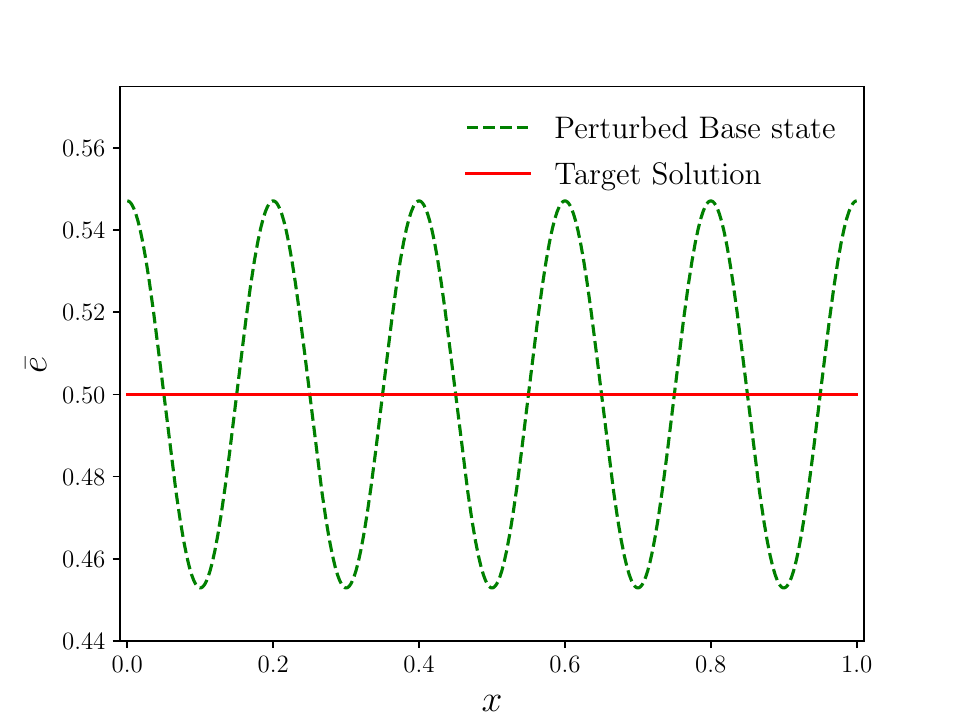}
        \caption{}
        \label{fig:case_2_sf_a}
    \end{subfigure}
    \begin{subfigure}[b]{0.45\textwidth}
        \centering
        \includegraphics[scale=.45]{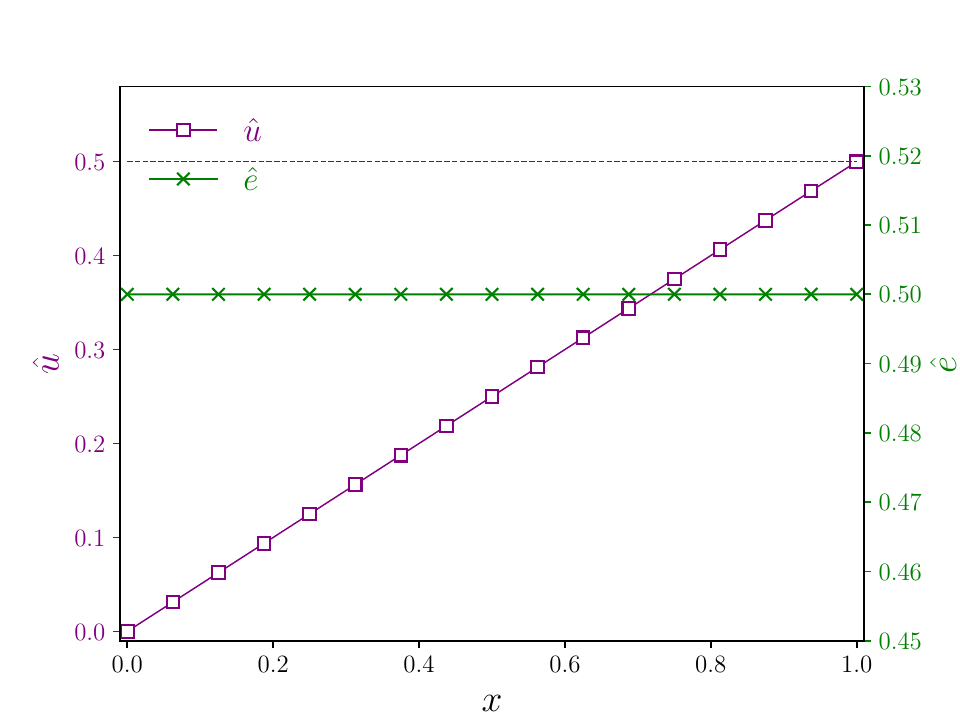}
        \caption{}
        \label{fig:case_2_sf_b}
    \end{subfigure}
    \caption{(a)  Base state $\bar{e}$ and (b) Solution for the primal fields.}
    \label{fig:case_2}
\end{figure}
\begin{table}[h]
\centering
\begin{tabular}{|c|c|c|} 
\hline
\textbf{Number of Elements} & \textbf{ $\|\hat{u}-u^{(t)}\|_1$} & {$\|\hat{e}-e^{(t)}\|_1$} \\
\hline
100 & $2\times10^{-4}$ & $1\times10^{-4}$\\
1600 & $1\times10^{-6}$ & $4\times10^{-7}$\\
8000 & $4\times10^{-8}$ & $1\times10^{-8}$\\
\hline
\end{tabular}
\caption{Error reduction in $L^1$ norms of displacement and strain with mesh refinement. These values are for the case corresponding to Fig.~\ref{fig:case_2}.}
\label{tab:case_2}
\end{table}
Next, we consider the energy functional
\begin{equation*}\label{eqn:case_1_func}
 I = \int_0^1 \left( \left( u_x - 1\right)^2 - 1 \right)^2 + (u - 0.5x)^2 \,dx; \qquad u(0) = 0; \quad u(1) = 1.
\end{equation*}
Clearly, a homogeneous strain solution is not possible in this case due to the mismatch $\alpha \neq \alpha^*$ and hence a solution is not easy to guess, but the expectation of the existence of a stressed, smooth solution of \eqref{eqn:primal_EL} seems eminently reasonable (even though, as usual, energy minimizers for this functional do not exist). However, a choice of a base state close to such an expected solution is not easy to guess. The equilibrium solution obtained in this case is shown in Fig.~\ref{fig:case_3_sf_b}. The difference between the base state chosen for this case and the solution can be observed in Fig.~\ref{fig:case_3_sf_a}.
 Convergence is obtained even with a $20\%$ difference in the average value of the initial guess for strain and the obtained solution ($e$). The problem is solved with varying mesh densities in the domain, starting with 8000 elements and comparing to solutions obtained with 4000, 2000, and 100 elements. The $L^1$ norm of the difference between solutions obtained with different mesh refinements is calculated and the results are presented in Table.~\ref{tab:case_3}, where $\hat{u}^{(n_e)}$ corresponds to the discrete displacement solution obtained using $n_e$ number of elements and similarly for the strain, $\hat{e}$. As the mesh is made finer the solutions appear to approach a fixed profile, this being interpreted as an indication of convergence. 

 We have considered the question of finding weak solutions to the equations of elastostatics \eqref{eqn:primal_EL} close to a guess for a limit of an energy minimizing sequence for this problem. The displacement profile for such a limit is expected to consist of a straight line of slope $\alpha = 0.5$ up to some point $x_0$ close to $x = 1$, followed by an  energy minimizing boundary layer profile satisfying the boundary conditions $u(x_0) = \alpha x_0$ and $u(0) = \alpha^* = 1$ (which can be obtained by solving the elastostatics equations with these b.c.s for fixed $x_0$, and then optimizing w.r.t $x_0$). However, since no considerations of force balance have entered, it is very unlikely that the traction jump condition is satisfied (the jump condition requires the discontinuity in stress, $\sigma$, in this 1-d case, to vanish at the displacement kink at $x_0$ joining the constant strain segment and boundary layer profile). Furthermore, the displacement profile for an element of such a sequence for $0 \leq x \leq x_0$ would consist of a typical `staircase' profile of slopes $0, 2$ along the backbone piece of slope $\alpha$, as shown in Fig.~\ref{fig:limit_seq}. However, due to the presence of the forcing term $(u - \alpha x)^2$ in the energy functional, constant strain segments do not satisfy the E-L equations \eqref{eqn:primal_EL} (while the jump condition would be satisfied at the interfaces of these `laminates'). Consequently, neither the elements of such an energy minimizing sequence, nor the limit can be expected to strictly satisfy the weak form of the E-L equations.

 On setting up the base state and initial guess as the limit displacement profile discussed above with $x_0 \neq \frac{2}{3}$, we have not been able to find solutions to the corresponding dual problem. On noting the fact that any continuous displacement profile comprising constant strain segments of slope $\{0,2\}$ does satisfy the traction jump condition, a staircase profile with these slopes riding on a slope $\half$ backbone up to $x_0 = \frac{2}{3}$ allows spanning the boundary layer with a constant strain of $2$. Thus, this is a profile for which traction jumps vanish, and with profiles close to this as a base state, the dual problem is able to find approximations to the weak form of the E-L equations. 
 
 In this manner, we have succeeded in finding many equilibrium solutions by the dual scheme close to a limit of a minimizing sequence, the limit in question being a continuous displacement profile with a single kink at $x_0 = \frac{2}{3}$, with constant slopes of $\half$ to the left and $2$ to the right of it. However, we have not succeeded in finding a dual (or a primal) weak solution to \eqref{eqn:primal_EL} using this limit profile as a base state and/or initial guess.
 \begin{figure}[H]
    \centering
        \includegraphics[scale=0.6]{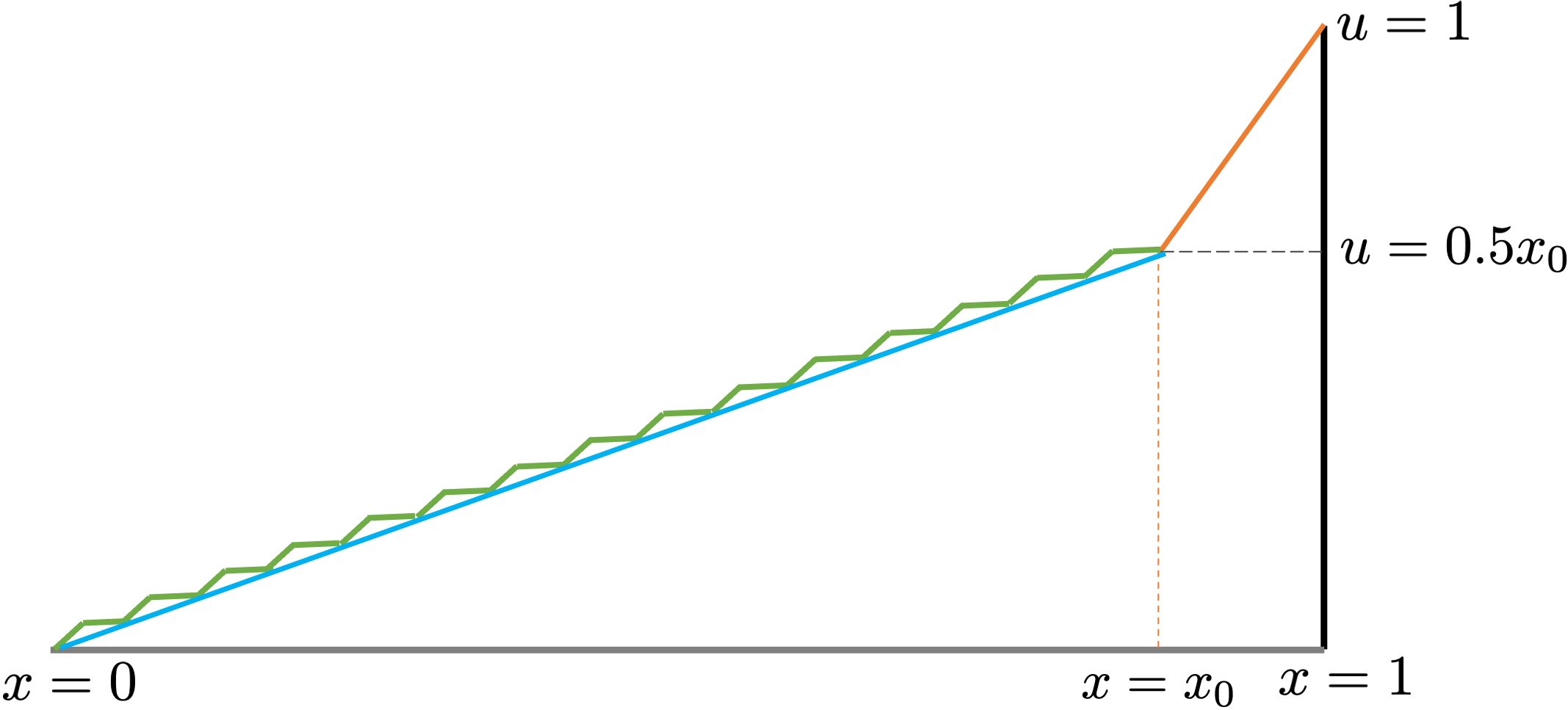}
        \caption{The displacement profile for the case $x_0 = \frac{2}{3}$. The staircase profile (green) rides on a line of slope $\frac{1}{2}$ (blue) until $x_0$, followed by an orange line of slope 2.}
        \label{fig:limit_seq}
\end{figure}
\begin{figure}[H]
    \centering
    \begin{subfigure}[b]{0.45\textwidth}
        \centering
        \includegraphics[scale=.45]{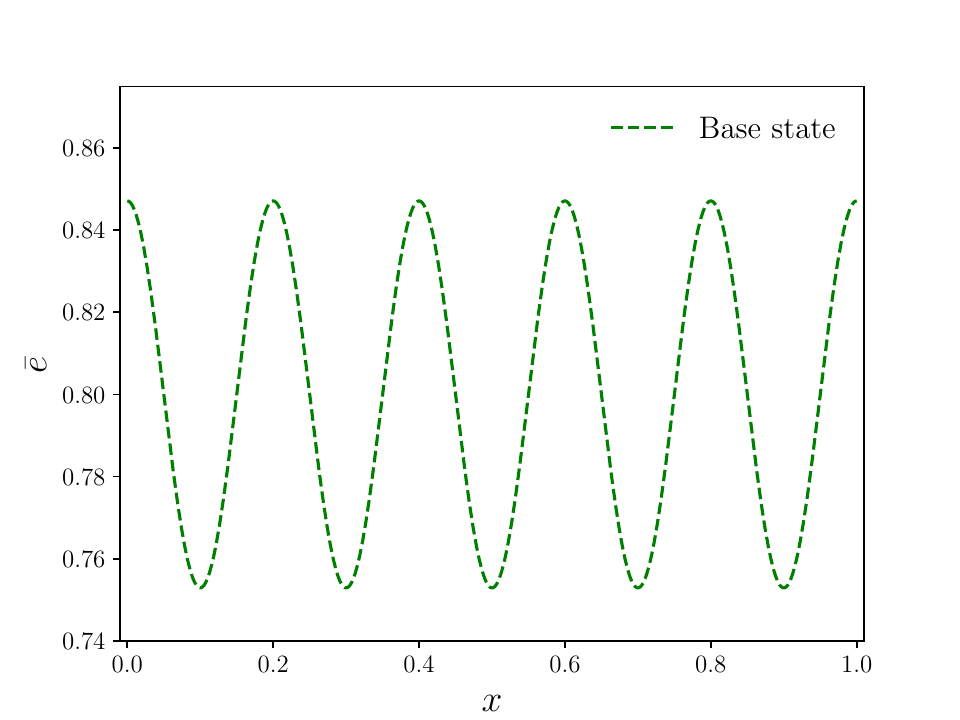}
        \caption{}
        \label{fig:case_3_sf_a}
    \end{subfigure}
    \begin{subfigure}[b]{0.45\textwidth}
        \centering
        \includegraphics[scale=.45]{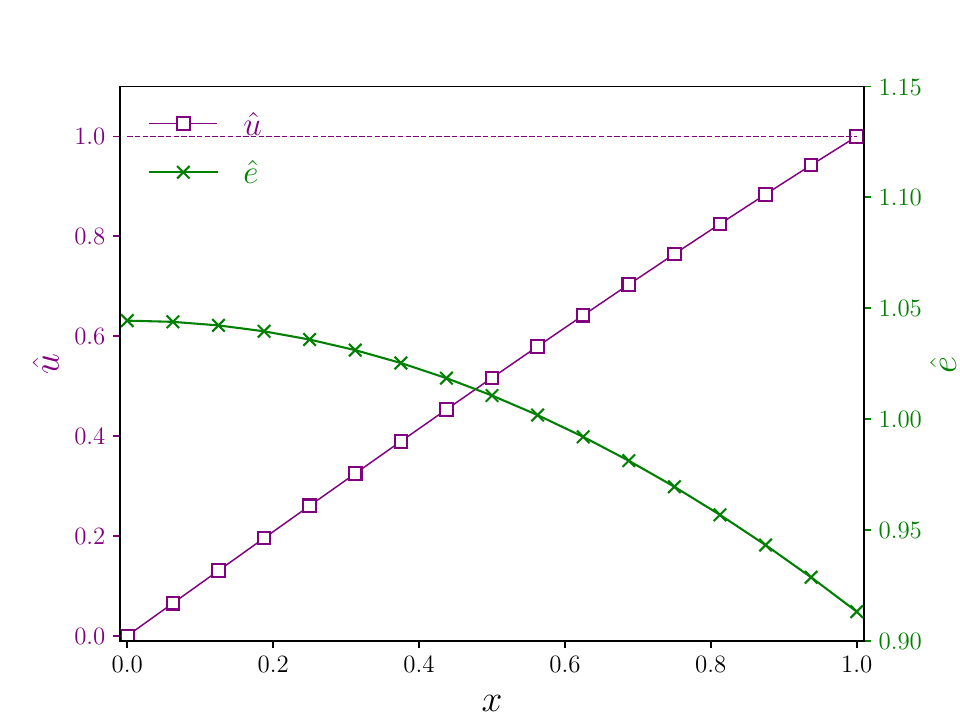}
        \caption{}
        \label{fig:case_3_sf_b}
    \end{subfigure}
    \begin{subfigure}[b]{0.45\textwidth}
        \centering%
        \includegraphics[scale=.45]{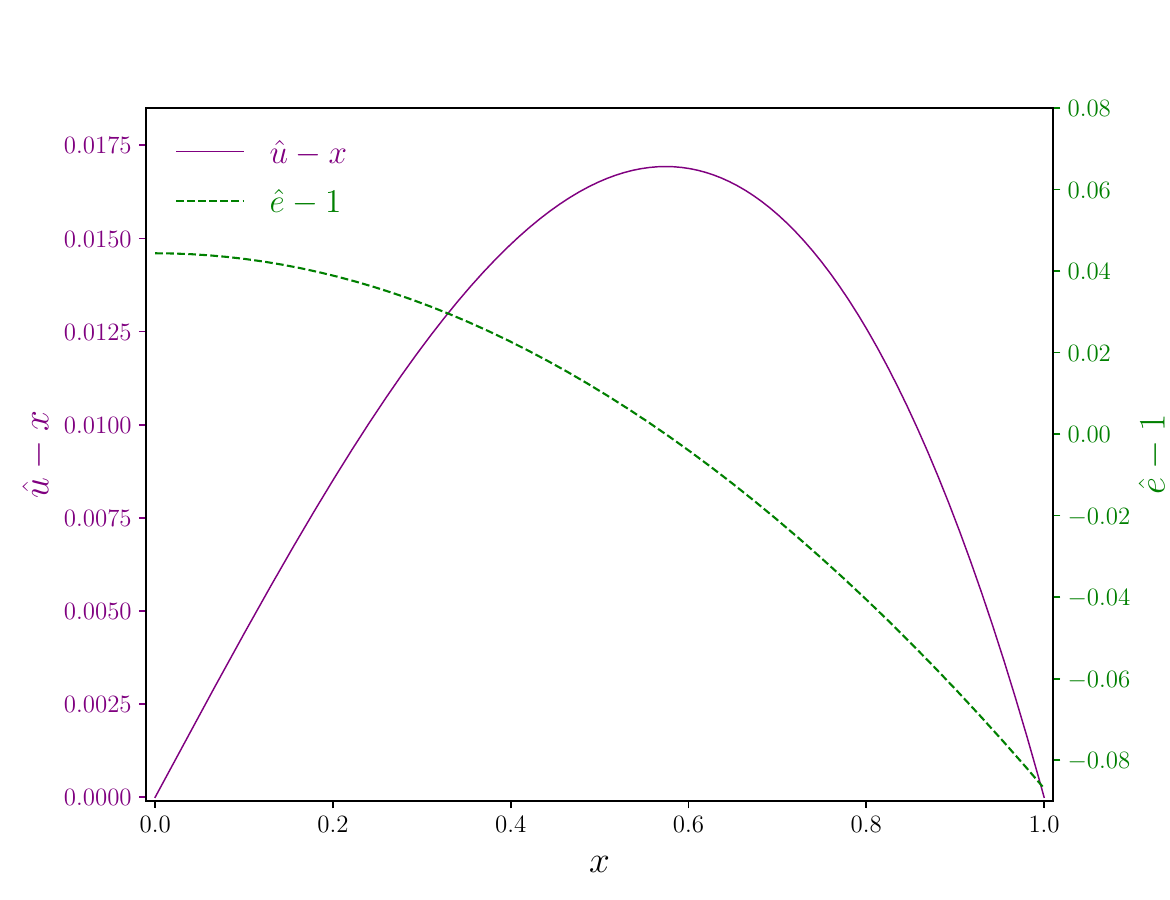}
        \caption{}
        \label{fig:case_3_sf_c}
    \end{subfigure}
    \caption{(a) Base state $\bar{e}$, (b)  Solution for the primal fields and (c)  Difference between the linear profile for $u$ and the constant profile for $e$ is shown for clarity.}
    \label{fig:case_3}
\end{figure}
 \begin{table}[H]
\centering
\begin{tabular}{|c|c|c|} 
\hline
\textbf{Number of Elements ($n_e$)} & \textbf{ $\|\hat{u}^{(n_e)}-\hat{u}^{(8000)}\|_1$} &{$\|\hat{e}^{(n_e)}-\hat{e}^{(8000)}\|_1$} \\
\hline
100 & $1\times10^{-5}$ & $2\times10^{-4}$\\
2000 & $4\times10^{-8}$ & $5\times10^{-7}$\\
4000 & $8\times10^{-9}$ & $1\times10^{-7}$\\
\hline
\end{tabular}
\caption{Mesh convergence for computations shown in Fig.~\ref{fig:case_3}.}
\label{tab:case_3}
\end{table}
\subsubsection{Tracking non-unique solutions through the choice of base states}
For the next two examples, we ignore the bulk loading term in the primal functional:
\begin{equation}\label{eqn:case_1_func}
 I = \int_0^1 \left( \left( u_x - 1\right)^2 - 1 \right)^2 \,dx;  \qquad u(0) = 0; \quad u(1) = 1\
\end{equation}
The goal is to show that we can recover distinct solutions in a seamless manner by using the choice of base states which acts to select particular solutions.
 As discussed in the preceding Sections \ref{sec:stab} and \ref{sec:deg_ell}, an important factor that determines the stability of a primal solution is the convexity/concavity of the functional at the obtained solution. We have plotted the energy density, ($\phi$), its derivative, the stress ($ \phi^\prime$), and its second derivative, the stiffness ($\phi^{\prime \prime}$) in Fig.~\ref{fig:expl_stab}. 
\begin{figure}[H]
    \centering
    \begin{subfigure}[b]{0.45\textwidth}
        \centering
        \includegraphics[scale=.45]{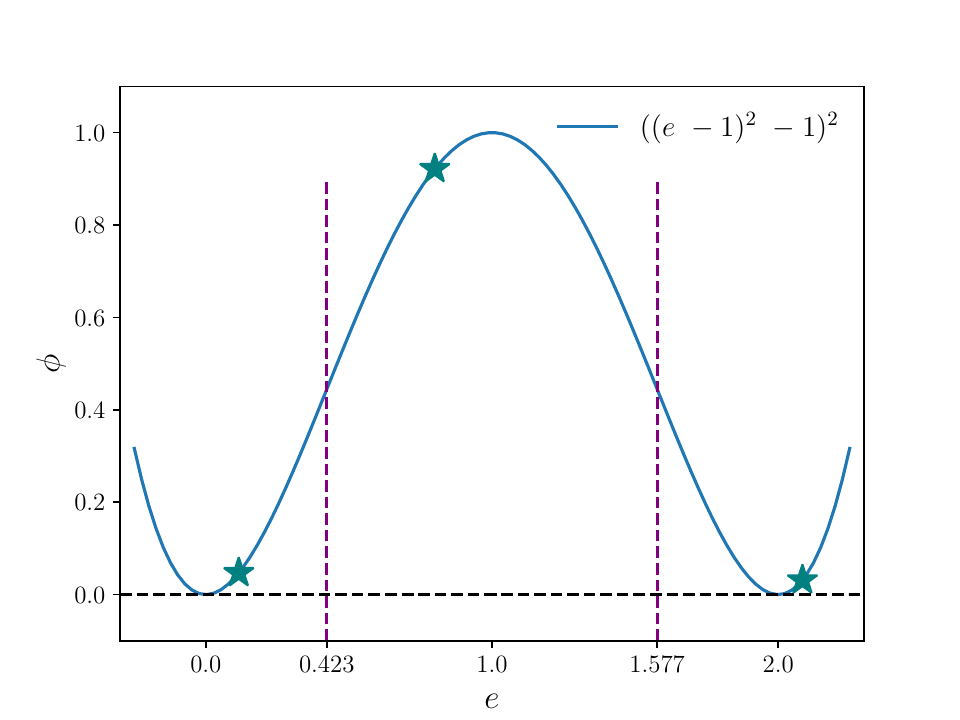}
        \caption{}
        \label{fig:expl_stab_ener}
    \end{subfigure}
    \begin{subfigure}[b]{0.45\textwidth}
        \centering
        \includegraphics[scale=.45]{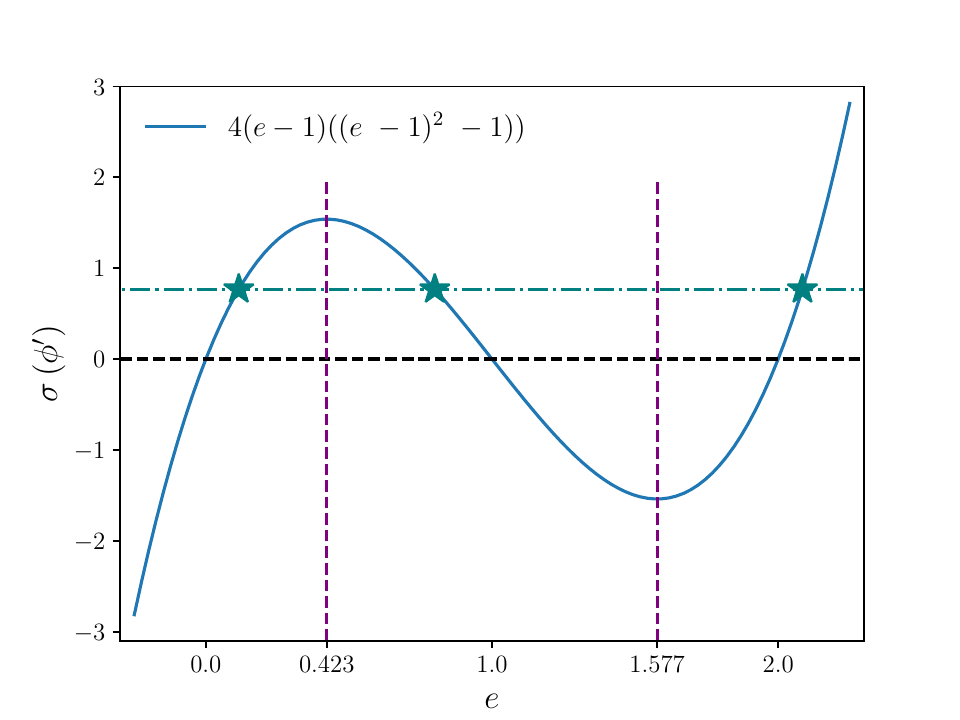}
        \caption{}
        \label{fig:expl_stab_stress}
    \end{subfigure}
    \begin{subfigure}[b]{0.45\textwidth}
        \centering%
        \includegraphics[scale=.45]{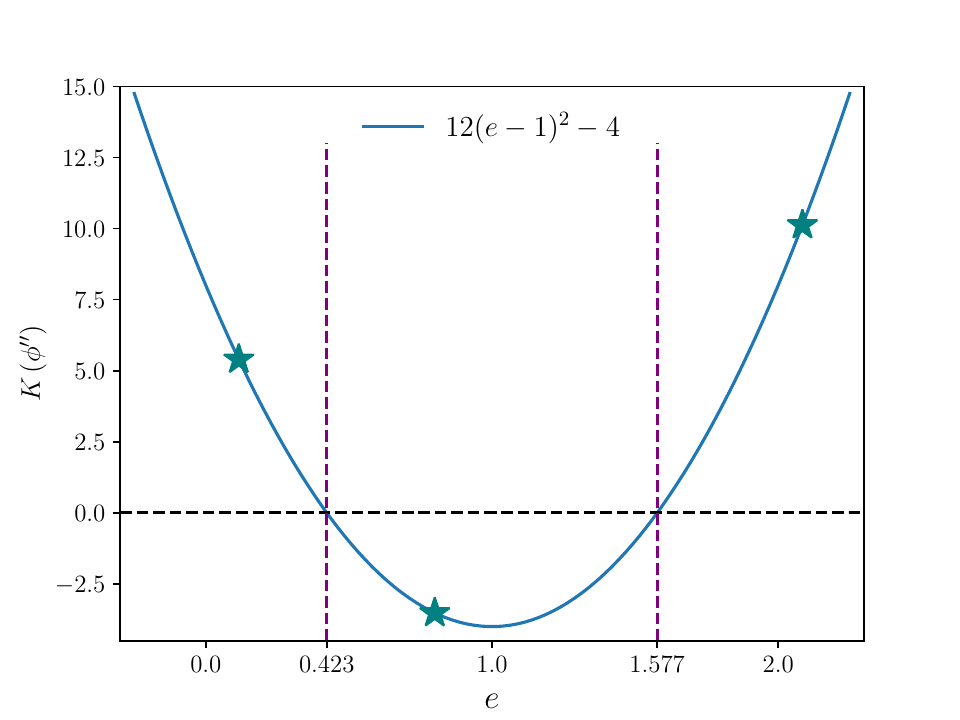}
        \caption{}
        \label{fig:expl_stab_stiff}
    \end{subfigure}
    \caption{Profile showing the (a)  energy density, (b)  its derivative (stress), and (c)  second derivative (stiffness). The region between the purple dashed lines, $0.423 \leq e \leq 1.577$, shows the strain range where stiffness is negative. The black dashed lines mark the zero on the y-axis. The $\star$s mark points that have the same stress, and are discussed in Sec.~\ref{sub_sec:case_5}.}
    \label{fig:expl_stab}%
\end{figure}
In this Section, we have obtained weak solutions to the primal E-L system \eqref{eqn:primal_EL} for which all points in the domain are either in the negative stiffness region or positive stiffness region using the dual scheme. Different solutions are obtained depending on the base state used. In one instance, the chosen base state and the obtained solution are shown in Fig.~\ref{fig:case_4_1}. Another solution for the given boundary condition is the `hat function' shown in Fig.~\ref{fig:case_4_2}. The base state provided in this example is of the form, $\bar{e} = 1 - a$ and $\bar{e} = 1 + a$. These two profiles are equally distributed in the different parts of the domain, as shown in Fig.~\ref{fig:case_4_1_sf_a} and Fig.~\ref{fig:case_4_2_sf_a}. Displacement boundary conditions are as in \eqref{eqn:case_1_func}. When $0\leq a \leq 0.3$, the former (uniform $e$ profile) solution is obtained; when  $0.9 \leq a \leq 5$, the latter (hat function profile) solution is obtained. Our scheme did not converge for $0.3 \leq a \leq 0.9$. Convergence with respect to mesh refinement for the two cases is shown in Tables.~\ref{tab:case_4_1} and \ref{tab:case_4_2}.
\begin{figure}[H]
    \centering
    \begin{subfigure}[b]{0.45\textwidth}
        \centering
        \includegraphics[scale=.45]{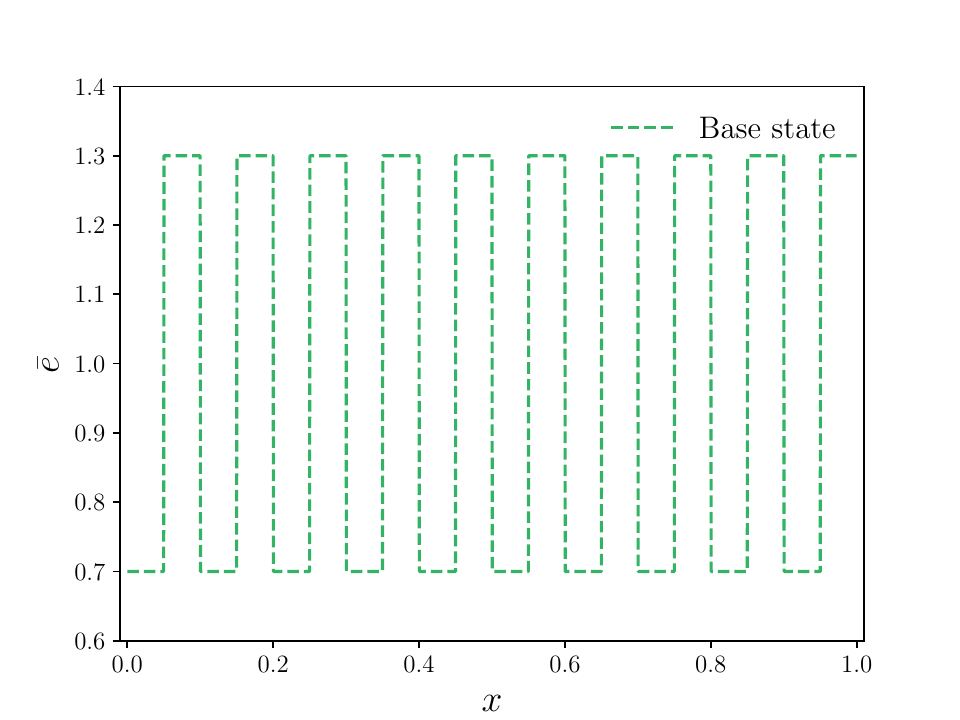}
        \caption{}
        \label{fig:case_4_1_sf_a}
    \end{subfigure}
    \begin{subfigure}[b]{0.45\textwidth}
        \centering
        \includegraphics[scale=.45]{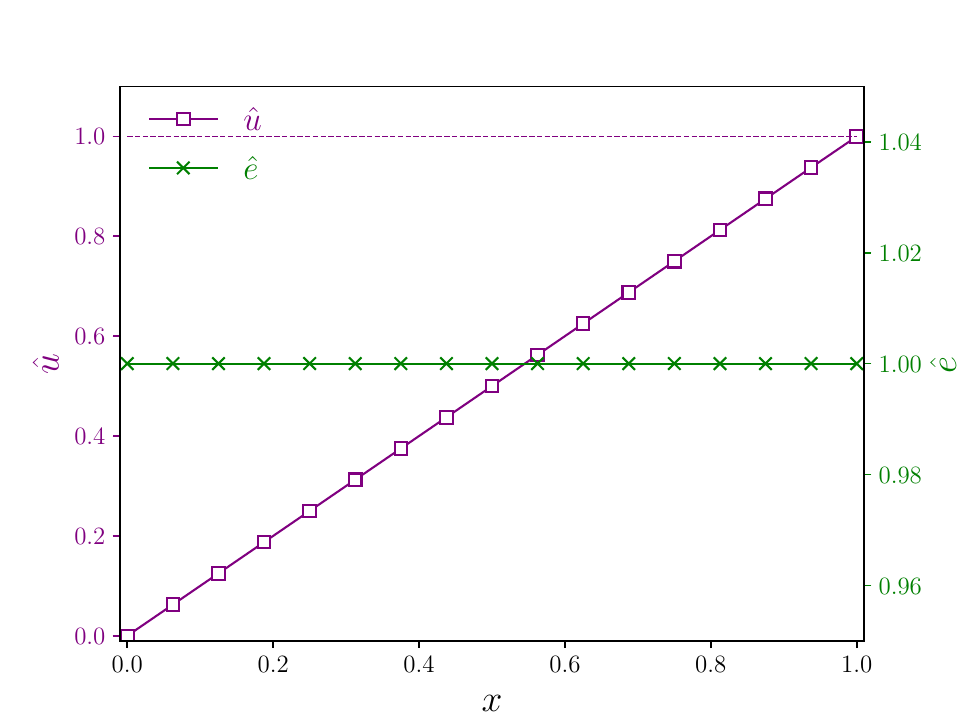}
        \caption{}
        \label{fig:case_4_1_sf_b}
    \end{subfigure}
    \caption{(a) Base state ($\bar{e}$) and (b) Solution for the primal fields.}
    \label{fig:case_4_1}%
\end{figure}
 \begin{table}[h]
\centering
\begin{tabular}{|c|c|c|} 
\hline
\textbf{Number of Elements ($n_e$)} & \textbf{ $\|\hat{u}^{(n_e)}-\hat{u}^{(8000)}\|_1$} &{$\|\hat{e}^{(n_e)}-\hat{e}^{(8000)}\|_1$} \\
\hline
100 & $3\times10^{-4}$ & $3\times10^{-8}$\\
2000 & $6\times10^{-7}$ & $9\times10^{-11}$\\
4000 & $1\times10^{-7}$ & $1\times10^{-11}$\\
\hline
\end{tabular}
\caption{Mesh convergence for computations shown in Fig.~\ref{fig:case_4_1}.}
\label{tab:case_4_1}
\end{table}
\begin{figure}[H]
    \centering
    \begin{subfigure}[b]{0.45\textwidth}
        \centering
        \includegraphics[scale=.45]{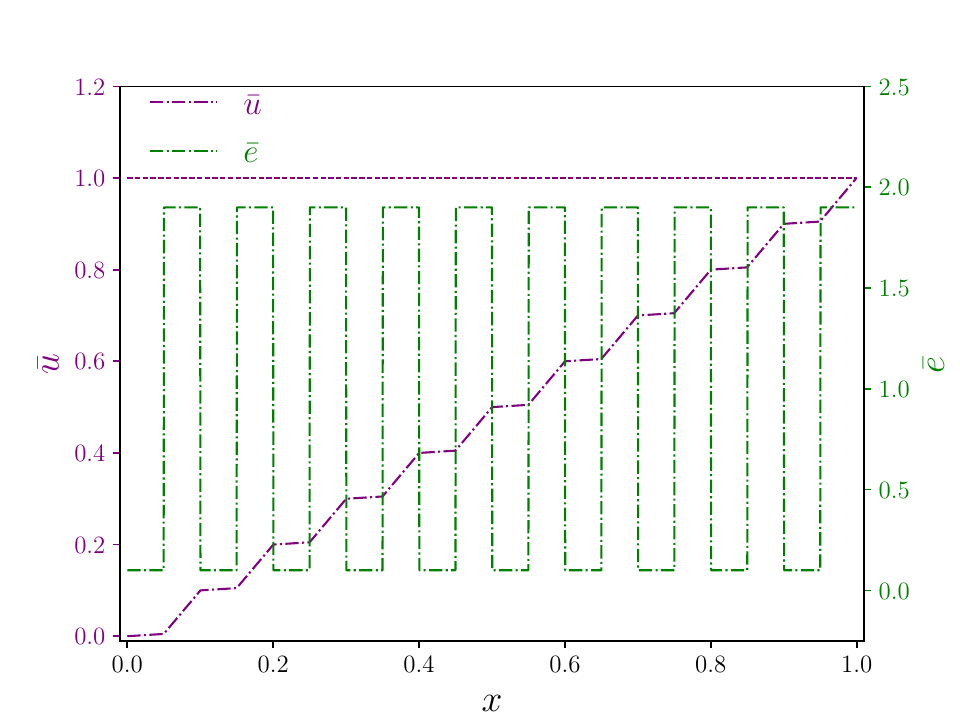}
        \caption{}
        \label{fig:case_4_2_sf_a}
    \end{subfigure}
    \begin{subfigure}[b]{0.45\textwidth}
        \centering
        \includegraphics[scale=.45]{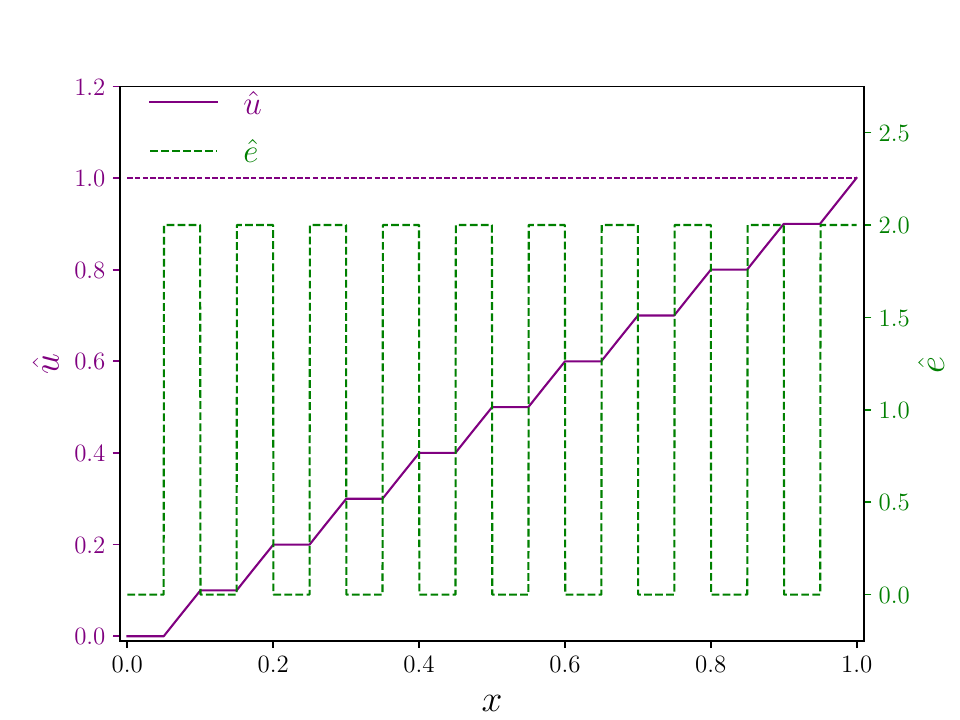}
        \caption{}
        \label{fig:case_4_2_sf_b}
    \end{subfigure}
    \caption{(a) Base state $\bar{u}$ and $\bar{e}$ and (b) Solution for the primal fields. This $\hat{u}$ field profile is referred to as hat function because of its shape.}
    \label{fig:case_4_2}%
\end{figure}
\begin{table}[h]
\centering
\begin{tabular}{|c|c|c|} 
\hline
\textbf{Number of Elements ($n_e$)} & \textbf{ $\|\hat{u}^{(n_e)}-\hat{u}^{(8000)}\|_1$} &{$\|\hat{e}^{(n_e)}-\hat{e}^{(8000)}\|_1$} \\
\hline
100 & $9\times10^{-5}$ & $4\times10^{-7}$\\
2000 & $2\times10^{-7}$ & $2\times10^{-8}$\\
4000 & $4\times10^{-8}$ & $1\times10^{-8}$\\
\hline
\end{tabular}
\caption{Mesh convergence for computations shown in Fig.~\ref{fig:case_4_2}.}
\label{tab:case_4_2}
\end{table}
\subsubsection{Stability of an equilibrium configuration with negative elastic stiffness grain boundaries without regularization: elastostatics and elastodynamics} \label{sub_sec:case_5}

We consider a 1-d idealization of an elastic body in the shape of a bar as shown in Fig.~\ref{fig:subfig_a}, comprising three grains separated by grain boundary phases. A high-angle grain boundary in an ordered medium is a strongly disordered region of the overall nominally crystalline state. We model the ordered grains as regions with positive elastic stiffness and the disordered grain boundary regions as ones with negative elastic stiffness. The material is described using the energy density profile shown in Fig.~\ref{fig:expl_stab}. As uniform stress over the domain is one of the solutions to the elastostatic problem, we look for a configuration with a constant stress field. Three strain ($e$) values (marked by $\star$ in Fig.~\ref{fig:expl_stab}) resulting in uniform stress over the domain are chosen in distinct parts of the domain, as shown in Fig.~\ref{fig:case_5_sf_b} and in Table~\ref{tab:targ_fig_elastostat}. Strain values in the stable region of the plots (positive stiffness), two of the three marked, are used to form the grains. The remaining value is used to form the unstable grain boundaries. The left boundary $(x = 0)$ is held fixed and the displacement b.c. at $x = 1$ is determined from integrating the assigned strain profile along $x$.
\begin{figure}[H]
    \centering
        \includegraphics[scale=.30]{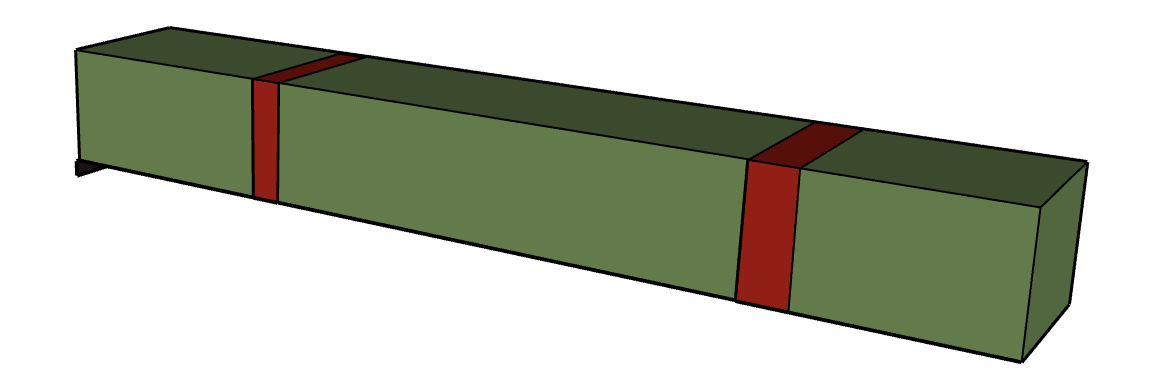}
    \caption{Schematic of a bar highlighting green grains separated by red grain boundaries. The left end is fixed and appropriate b.c.s are applied to the right end.} 
     \label{fig:subfig_a}
\end{figure}
\begin{table}[H]
    \centering
    \begin{tabular}{|c|c|c|}
\hline
Domain & $e^{(t)}$ & $u^{(t)}$ \\ \hline
$0 \leq x < 0.3225$ & $0.115$ & $0.115x$ \\ \hline
$0.3225 \leq x < 0.3325$ & $0.8$ & $0.8(x - 0.3225) + 0.037$ \\ \hline
$0.3325 \leq x < 0.8275$ & $2.085$ & $0.045 + 2.085(x - 0.3325)$ \\ \hline
$0.8275 \leq x < 0.8875$ & $0.8$ & $1.077 + 0.8(x - 0.8275)$ \\ \hline
$0.8875 \leq x \leq 1$ & $0.115$ & $1.125 + 0.115(x - 0.8875)$ \\ \hline
\end{tabular}
    \caption{Target solution for the case presented in Fig.~\ref{fig:case_5}.}
    \label{tab:targ_fig_elastostat}
\end{table}

{\bf Elastostatics:} The base state chosen is as shown in Fig.~\ref{fig:case_5}, for the $\bar{e}$ (and consistent $\bar{u}$) and, following standard protocol for solving the dual problem, we start from an initial guess coincident with it. The applied boundary condition is $\hat{u}(0) = 0$ and $\hat{u}(1) = 1.138$. The target solution is listed in Table~\ref{tab:targ_fig_elastostat}. The result, Fig.~\ref{fig:case_5}, shows that the solution obtained by the dual scheme has a significant `radius' of its basin of attraction, as gauged by the deviation of the initial guess (and base state) from the obtained solution. A convergence study w.r.t mesh refinement is also performed on the dual solution and presented in Table.~\ref{tab:case_5}.
\begin{figure}
    \centering
    \begin{subfigure}[b]{0.45\textwidth}
        \centering
        \includegraphics[scale=.45]{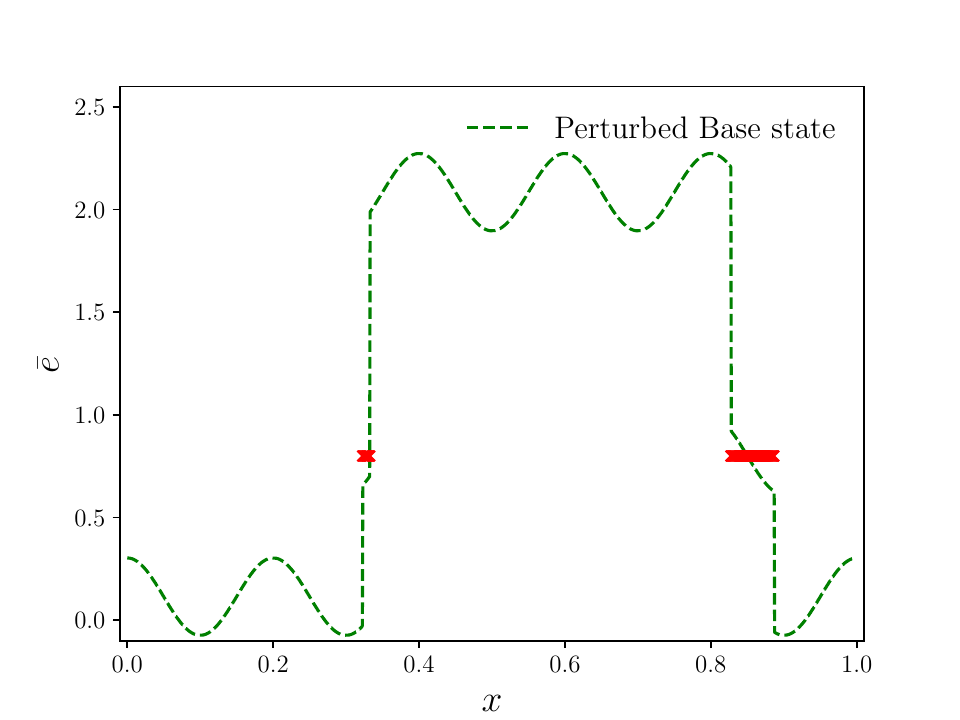}
        \caption{}
        \label{fig:case_5_sf_a}
    \end{subfigure}
    \begin{subfigure}[b]{0.45\textwidth}
        \centering
        \includegraphics[scale=.45]{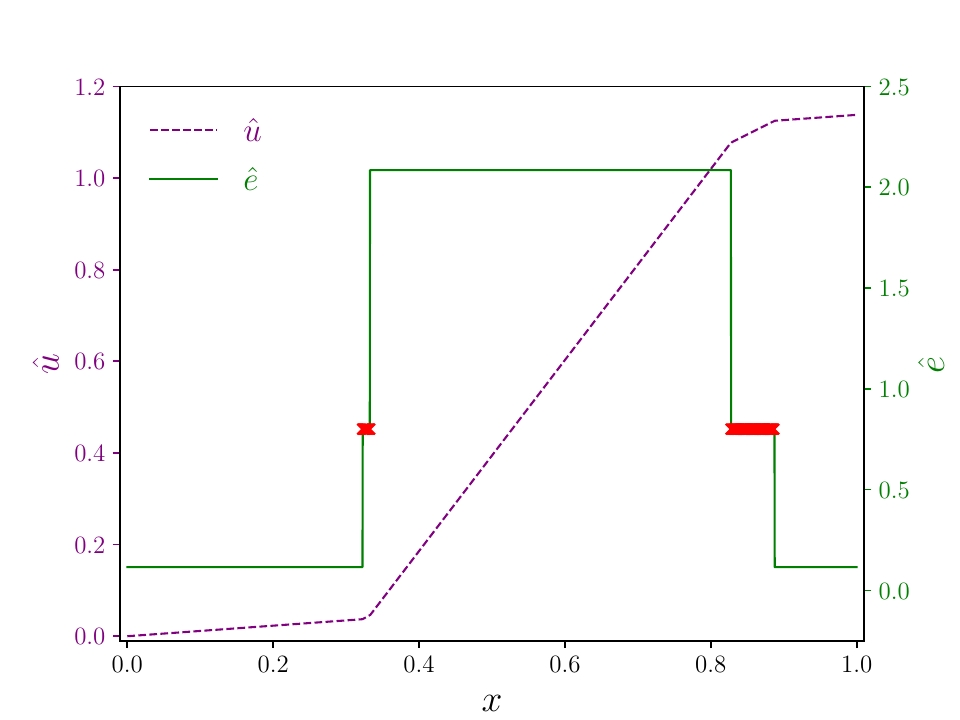}
        \caption{}
        \label{fig:case_5_sf_b}
    \end{subfigure}
    \caption{(a) Base states $\bar{e}$ field and (b) shows solution for the primal fields. Grain boundaries are marked with red $\star$.}
    \label{fig:case_5}%
\end{figure}
\begin{table}[h]
\centering
\begin{tabular}{|c|c|c|} 
\hline
\textbf{Number of Elements} & \textbf{ $\|\hat{u}-\hat{u}^{(t)}\|_1$} & {$\|\hat{e}-\hat{e}^{(t)}\|_1$} \\
\hline
100 & $2\times10^{-5}$ & $8\times10^{-7}$\\
1600 & $9\times10^{-8}$ & $6\times10^{-8}$\\
8000 & $1\times10^{-8}$ & $3\times10^{-9}$\\
\hline
\end{tabular}
\caption{Error Reduction in $L^1$ norms of displacement and strain with mesh refinement. These values are for the case corresponding to Fig.~\ref{fig:case_5}.}
\label{tab:case_5}
\end{table}

{\bf Elastodynamics:} We probe the dynamic stability of the equilibrium solution obtained in the elastostatic case above. As is well-understood, regions corresponding to  negative stiffness, as evaluated from the initial condition for the dynamic problem, are locally elliptic and hence the primal problem is ill-posed as an initial value problem and suffers from the `Hadamard instability' (or lack of continuous dependence w.r.t initial data). Solved by the dual scheme, however, this problem, which has a reasonable physical interpretation as mentioned above, is approachable as a stable one, i.e. thinking of the dual elastodynamic scheme as a mathematical model for a body with grain boundaries, represented here as containing regions with negative stiffness, an initial condition corresponding to an elastostatic solution to \eqref{eqn:primal_ED} remains stable under numerical evolution without further loading. In contrast, the primal formulation of elastodynamics suffers from severe numerical instability, resulting in (inevitably present) round-off errors from computer arithmetic `blowing-up' in finite (very short) time (as well-understood, and expected).

The theory discussed in Sec.~\ref{sec:elastodyn} is employed in order to solve the elastodynamic problem. In  1 space dimension compatibility is not a constraint, and we consider only two primal variables $e$ and $v$, corresponding to $F$ and $v_i$, respectively. The system of equations solved, therefore, becomes
\begin{subequations}
    \label{eqn:primal_ED}
    \begin{align}
        \sigma(e)_x - \rho_0 v_t & = 0 \qquad \mbox{in} \qquad \Omega \times (0,T) \\
        e_t - v_x & = 0 \qquad \mbox{in} \qquad \Omega \times (0,T) \\
        v(x,0) = v^{(0)}(x)  ; \qquad  e(x,0) = e^{(0)}(x)  &\qquad v(0,t) = v^{(l)}(t)  ; \qquad  v(1,t) = v^{(r)}(t).
     \end{align}
\end{subequations}
The base state chosen for $e$ is a shifted cubic and for $v$ a shifted quadratic: 
\begin{equation*}
H(\hat{v}, \hat{e}; x,t) = \frac{1}{2}\left( c_v(\hat{v}-\bar{v}(x,t))^2\right) +\frac{1}{2}\left(c_e(\hat{e} - \bar{e}(x,t))^2\right)
+\frac{1}{3}\left(c_e(\hat{e} - \bar{e}(x,t))^3\right).
\end{equation*}
The DtP mapping equations and the corresponding dual functional, using the dual (Lagrange multiplier) fields $L$ and $P$, are given by
\begin{equation}
    \label{eq:Dtp_ED}
    \begin{aligned}
        \hat{v} &= \frac{ \rho_0 L_t - P_x}{c_v} + \bar{v} \\
        c_{e}\frac{\hat{e}-\bar{e}}{|\hat{e}-\bar{e}|} &\left(\hat{e} - \bar{e}\right)^2 + c_e\left(\hat{e} - \bar{e}\right) - P_t + L_x\left(\frac{d\sigma}{de}\Big|_{\hat{e}}\right)  = 0
    \end{aligned}
\end{equation}
\begin{equation}
\label{eqn:dual_exact_dyn}
 \begin{aligned}
     S\, \left[L, P \right] \,&=\, \int_0^T\int_0^1\,\, dx \, dt \left[\left(-\rho_0 L_t \hat{v} + \sigma L_x\right)  +\left(P_x \hat{v} - \hat{e} P_t\right) + H\left(\hat{v},\hat{e}\right)\right]\\
     &+ \int_0^T \left(Pv^{(l)}\right)dt - \int_0^T \left(Pv^{(r)}\right)dt - \int_0^1 \left(Pe^{(0)}\right)dx - \int_0^1 \left(\rho_0 L v^{(0)}\right)dt. 
     \end{aligned}
\end{equation}
The numerical framework for solving the elastostatic problem discussed in Sec.~\ref{subsec:alg_stat} extends seamlessly to solve the dual elastodynamics problem. 
 The latter is a 2-d boundary value problem (in space-time), having $x$ and $t$ as the two dimensions. All primal boundary and initial conditions are imposed. At any boundary where boundary/initial conditions on a primal variable is not available, we apply homogeneous zero Dirichlet boundary condition on the corresponding dual variable. The residual for this case can again be obtained by taking the first variation of the dual functional and the second variation yields the Jacobian. The Newton-Raphson scheme is used to solve this system as well.

To analyze the stability of the elastostatic solution obtained above in this section, we provide that solution (see, Table~\ref{tab:targ_fig_elastostat}) as an initial condition. Zero velocity boundary conditions are imposed, i.e. $v^{(l)}(t) = v^{(r)}(t) = 0$. This is equivalent to evolving the system without any external influence. We evolve this case, shown in Fig.~\ref{fig:elast_d_gb}, using both a  primal scheme\footnote{The primal scheme refers to the conventional second-order form of \eqref{eqn:primal_ED} solved by a standard Galerkin discretization in space and an explicit centered difference scheme in time with very conservative time stepping.} and the dual scheme. It is observed in Fig.~\ref{fig:elast_d_gb_sf_b} that the maximum strain magnitude ($|e|$) becomes unbounded (`NaN' in computer output) within small times when evolved using the primal scheme. This unbounded behaviour is not observed when the problem is evolved using the dual scheme (Fig.~\ref{fig:elast_d_gb_sf_c}); it stays stable at the initial equilibrium configuration even after a significant duration of time.

\begin{figure}[H]
    \centering
    \begin{subfigure}[b]{0.51\textwidth}
        \centering
        \includegraphics[scale=.51]{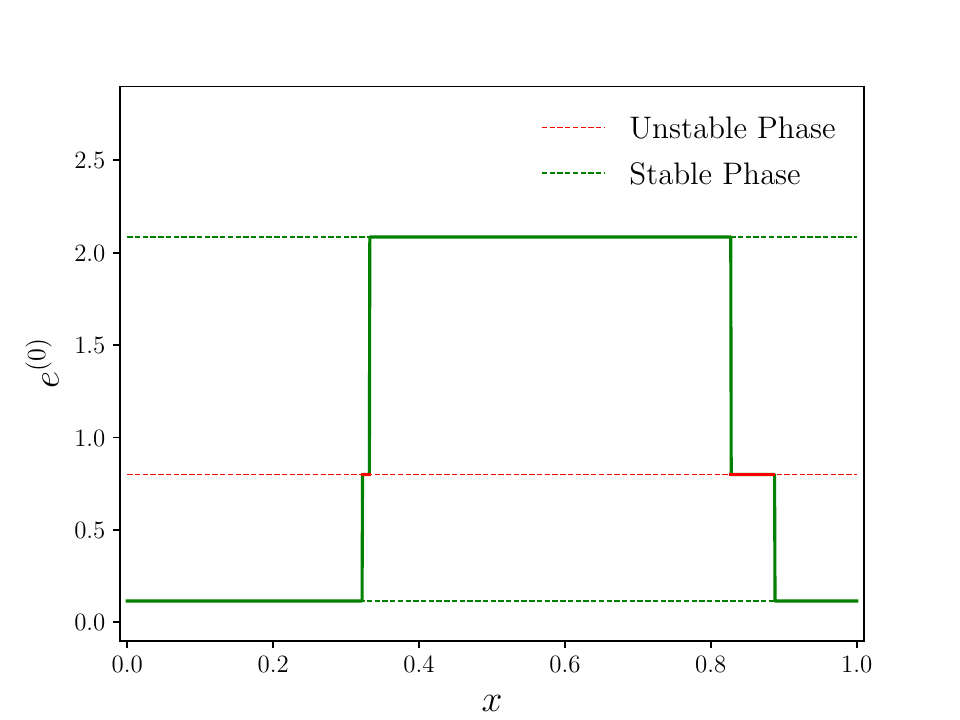}\\
        \caption{}
        \label{fig:elast_d_gb_sf_a}
    \end{subfigure}
    \begin{subfigure}[b]{0.49\textwidth}
        \centering
        \includegraphics[scale=.49]{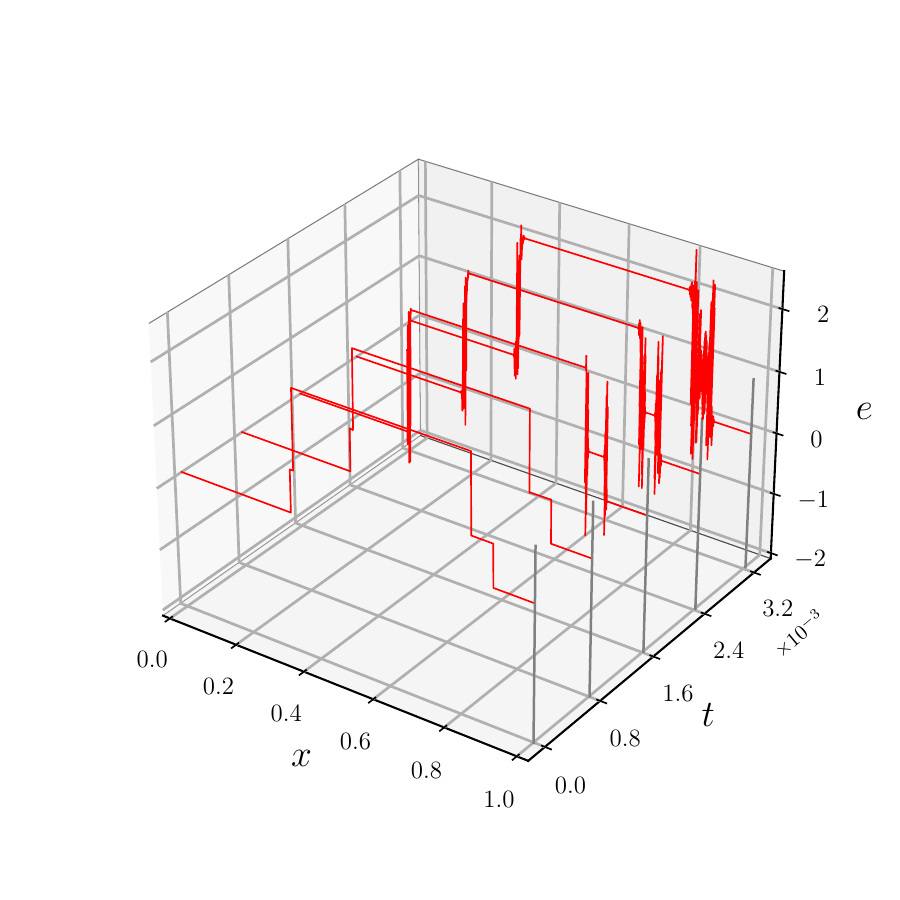}
        \caption{}
        \label{fig:elast_d_gb_sf_b}
    \end{subfigure}
    \begin{subfigure}[b]{0.49\textwidth}
        \centering%
        \includegraphics[scale=.49]{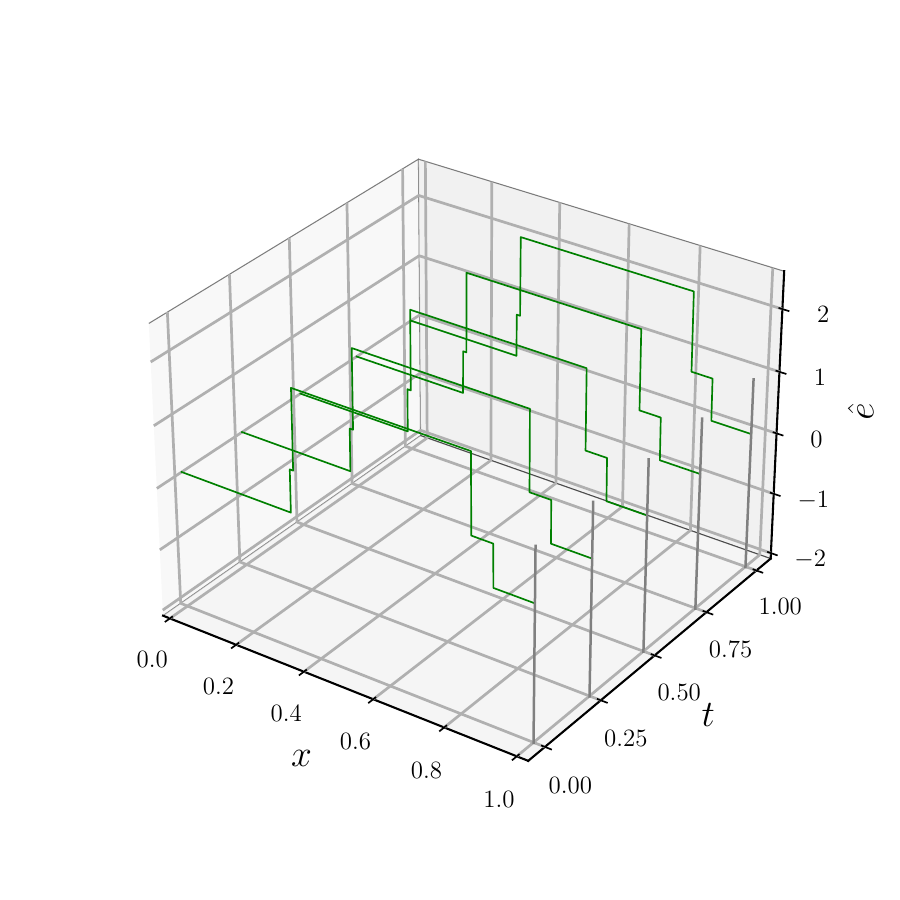}
        \caption{}
        \label{fig:elast_d_gb_sf_c}
    \end{subfigure}
    \caption{(a) Initial condition ($e$) for both primal and dual scheme. Over time, the solution to the primal problem diverges (red profile) as shown in (b), becoming unbounded (`NaN') after the next few time steps. Evolution solved using the dual scheme, shown in (c), remains stable (green profile). Figures (b) and (c) are plotted for different time intervals.}
    \label{fig:elast_d_gb}%
\end{figure}

We next probe the evolution of a small perturbation (i.e.~within the typical elastic limit) of this equilibrium solution in the dual formulation induced by a perturbation of the initial condition on strain and a corresponding perturbation in the displacement profile. The initial condition for the displacement provided w.r.t the equilibrium solution (Fig.~\ref{fig:case_5_sf_b}) is shown in Fig.~\ref{fig:case_6_sf_a}. The perturbations are provided in the stable phases of the domain. As expected, the initial profile breaks up into two waves of approximately half the amplitude of the initial displacement perturbation moving in opposite directions, following the trend of a D'Alembert-like solution from linear elastodynamics. The wave speed is observed to be higher in the region of higher stiffness.
\begin{figure}
    \centering
    \begin{subfigure}[b]{0.45\textwidth}
        \centering
        \includegraphics[scale=.40]{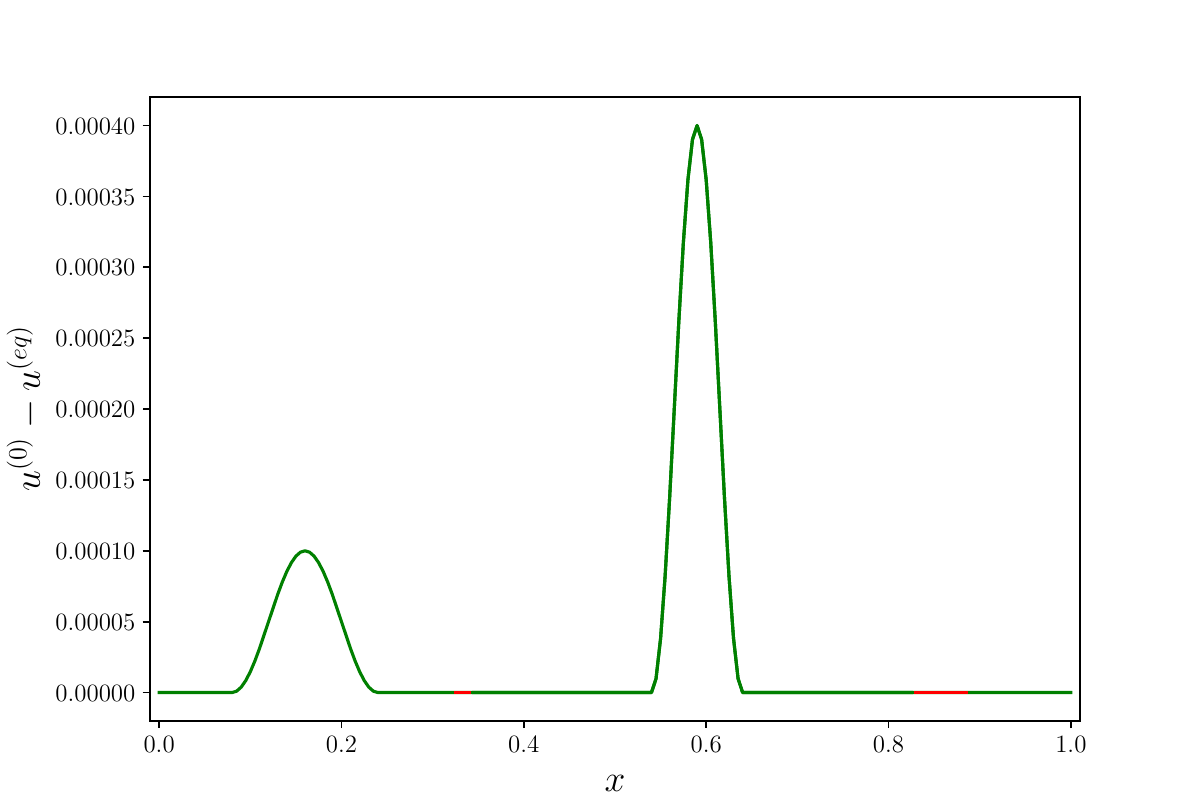}
        \caption{}
        \label{fig:case_6_sf_a}
    \end{subfigure}
    \begin{subfigure}[b]{0.45\textwidth}
        \centering
        \includegraphics[scale=.40]{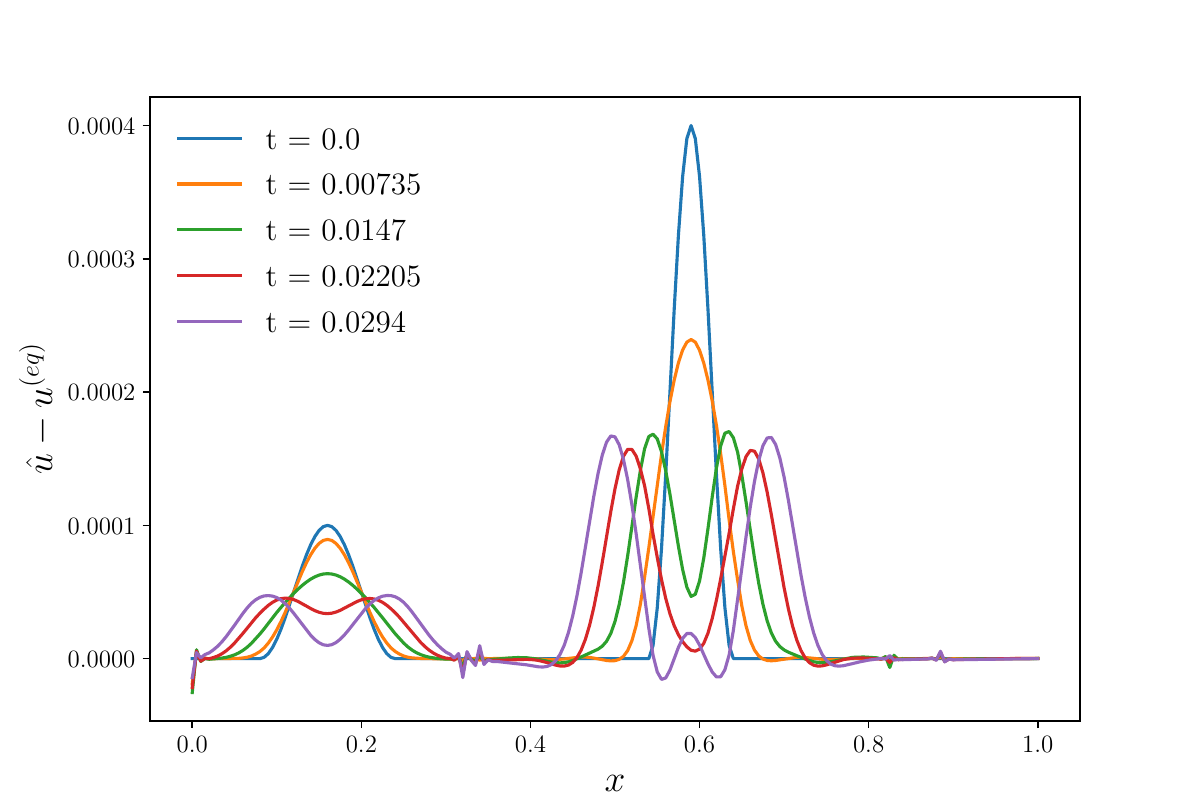}
        \caption{}
        \label{fig:case_6_sf_b}
    \end{subfigure}
    \caption{Evolution of displacement profiles relative to equilibrium: (a) Initial perturbations, (b) Displacement behavior over time. $u^{(0)}$, $u^{(eq)}$ represents initial and equilibrium profile of the displacement. $t$ corresponds to time stamp. Green and red color in (a) represents part of the domain in stable and unstable phase, respectively.}
    \label{fig:case_6}%
\end{figure}

\section{Conclusion}\label{sec:concl}
We have demonstrated a methodology for developing a (family of) dual variational principle(s) for nonlinear elastostatics and elastodynamics with the property that its Euler-Lagrange equations are the governing PDE of nonlinear elasticity with the primal fields represented in terms of a mapping of the dual fields and their space-time gradient, pointwise. The dual variational principles are convex by design, regardless of the properties of the primal problem. Existence of minimizers of such a dual variational principle is rigorously shown, and the notions of variational dual solutions and dual solution to the PDE of nonlinear elastostatics are defined in terms of the dual functional. Computed examples of illustrative cases corresponding to noncovex elastostatics and elastodynamics, including a stable dual calculation of an ill-posed, primal elastodynamic problem is also demonstrated.

The computed examples, especially in Sec.~\ref{sub_sec:case_5}, raise a philosophical question about the optimal, minimalist, mathematical model that should be adopted for modeling physically observed softening behavior in macroscopic mechanical response. This work seems to suggest that existence of dual solutions may not be a show-stopper, with the means of attaining uniqueness in the model transferred to the choice of special types of auxiliary functions $H$ employing physically meaningful base states which, nevertheless, need not be solutions to singularly-perturbed higher order PDE when physical justification for such is not available, but include experimentally observed input on physically relevant length scales. A crucial test in this regard would be to study the details of wave `propagation' through the non-hyperbolic regions in our model, and check how well such behavior squares with experiment.

\section*{Acknowledgment}
We thank Robert V. Kohn for helpful discussion and suggestions. This work was partially done while AA was on sabbatical leave in 2023 at a) the Max Planck Institute for Mathematics in the Sciences in Leipzig, and b) the Hausdorff Institute for Mathematics at the University of Bonn funded
by the Deutsche Forschungsgemeinschaft (DFG, German Research Foundation) under Germany’s
Excellence Strategy – EXC-2047/1 – 390685813, as part of the Trimester Program on Mathematics
for Complex Materials. The support and hospitality of both institutions are acknowledged. The collaboration between JG and AA on this paper was initiated at the trimester program in Bonn. It is also a pleasure to acknowledge discussions at these locations on the subject matter of this paper with John M.~Ball, Lukas Koch, and Angkana R\"uland. AA was supported by a Simons Pivot Fellowship grant \# 983171. SS was supported by the NSF OIA-DMR grant \# 2021019 and the Center for Extreme Events in Structurally Evolving Materials, ARL contract \# W911NF2320073. 

\bibliographystyle{alphaurl}\bibliography{main.bib}
\end{document}